\documentclass[11pt,times,twoside]{article}

\usepackage[leqno]{amsmath}
\usepackage{graphicx}
\usepackage{latexsym}
\usepackage{amsmath,amsfonts,amssymb,amsthm,euscript,makeidx,color,mathrsfs}
\usepackage{enumerate}
\usepackage[colorlinks,linkcolor=blue,anchorcolor=green,citecolor=red]{hyperref}

\def\sqr#1#2{{\vcenter{\vbox{\hrule height.#2pt
              \hbox{\vrule width.#2pt height#1pt \kern#1pt \vrule width.#2pt}
              \hrule height.#2pt}}}}
\def\signed #1{{\unskip\nobreak\hfil\penalty50
              \hskip2em\hbox{}\nobreak\hfil#1
              \parfillskip=0pt \finalhyphendemerits=0 \par}}
\def\endpf{\signed {$\sqr69$}}

\def\5n{\negthinspace \negthinspace \negthinspace \negthinspace \negthinspace }
\def\4n{\negthinspace \negthinspace \negthinspace \negthinspace }
\def\3n{\negthinspace \negthinspace \negthinspace }
\def\2n{\negthinspace \negthinspace }
\def\1n{\negthinspace }

\def\dbE{\mathbb{E}}
\def\dbF{\mathbb{F}}

\def\dbH{\mathbb{H}}

\def\dbP{\mathbb{P}}

\def\dbR{\mathbb{R}}

\def\sT{\mathscr{T}}

\def\of{\overline{f}}

\def\ds{\displaystyle}

\def\ns{\noalign{\ss}}
\def\a{\alpha}
\def\b{\beta}
\def\g{\gamma}
\def\d{\delta}
\def\e{\varepsilon}
\def\z{\zeta}

\def\l{\lambda}

\def\t{\tau}
\def\f{\varphi}
\def\th{\theta}
\def\o{\omega}
\def\p{\phi}
\def\i{\infty}
\def\G{\Gamma}
\def\D{\Delta}

\def\Om{\Omega}
\def\cA{{\cal A}}
\def\cB{{\cal B}}
\def\cC{{\cal C}}

\def\cF{{\cal F}}

\def\cZ{{\cal Z}}

\def\hY{{\hat Y}}
\def\hZ{{\hat Z}}

\def\hy{{\hat y}}
\def\hz{{\hat z}}
\def\bY{\bar{Y}}
\def\bZ{\bar{Z}}
\def\ba{\bar{a}}

\def\bc{\bar{c}}

\def\be{\bar{e}}

\def\bl{\bar{l}}

\def\bp{\bar{p}}

\def\br{\bar{r}}

\def\bt{\bar{t}}

\def\by{\bar{y}}
\def\bz{\bar{z}}

\def\no{\noindent}

\def\ss{\smallskip}
\def\ms{\medskip}

\def\q{\quad}
\def\qq{\qquad}
\def\hb{\hbox}

\def\ra{\mathop{\rightarrow}}

\def\rf{\eqref}
\def\sgn{\mathop{\rm sgn}}

\def\esssup{\mathop{\rm esssup}}

\def\cd{\cdot}

\def\sgn{\hbox{\rm sgn$\,$}}

\def\deq{\triangleq}
\def\les{\leqslant}
\def\ges{\geqslant}

\def\({\Big (}
\def\){\Big )}
\def\[{\Big[}
\def\]{\Big]}

\def\bde{\begin{definition}\label}
\def\ede{\end{definition}}
\def\be{\begin{equation}}
\def\bel{\begin{equation}\label}
\def\ee{\end{equation}}
\def\bt{\begin{theorem}\label}
\def\et{\end{theorem}}
\def\bc{\begin{corollary}\label}
\def\ec{\end{corollary}}
\def\bl{\begin{lemma}\label}
\def\el{\end{lemma}}
\def\bp{\begin{proposition}\label}
\def\ep{\end{proposition}}
\def\bas{\begin{assumption}\label}
\def\eas{\end{assumption}}
\def\br{\begin{remark}\label}
\def\er{\end{remark}}
\def\bex{\begin{example}\label}
\def\ex{\end{example}}
\def\ba{\begin{array}}
\def\ea{\end{array}}
\def\ed{\end{document}}
\def\ts{\times}
\def\olb\
\def\eps{\epsilon}
\def\square#1{\vbox{\hrule\hbox{\vrule height#1%
     \kern#1\vrule}\hrule}}
\def\rectangle#1#2{\vbox{\hrule\hbox{\vrule height#1%
     \kern#2\vrule}\hrule}}

\font\tenbb=msbm10 \font\sevenbb=msbm7 \font\fivebb=msbm5

\newfam\bbfam
\scriptscriptfont\bbfam=\fivebb \textfont\bbfam=\tenbb
\scriptfont\bbfam=\sevenbb

\newtheorem{theorem}{Theorem}[section]
\newtheorem{definition}[theorem]{Definition}
\newtheorem{proposition}[theorem]{Proposition}
\newtheorem{corollary}[theorem]{Corollary}
\newtheorem{lemma}[theorem]{Lemma}
\newtheorem{remark}[theorem]{Remark}
\newtheorem{example}[theorem]{Example}

\newtheorem{assumption}[theorem]{Assumption}

\makeatletter
   
   \@addtoreset{equation}{section}
\makeatother

\allowdisplaybreaks[4]

\begin{document}

\title{\bf Anticipated backward stochastic differential equations with quadratic growth\thanks{This work is partial supported by Lebesgue Center of Mathematics ``Investissements d'avenir" program-ANR-11-LABX-0020-01, by CAESARS-ANR-15-CE05-0024 and by MFG-ANR16-CE40-0015-01, by Research Grants Council of Hong Kong under grants 15255416 and 15213218, by National Natural Science Foundation of China (Grant Nos. 11871309, 11371226).}}

\author{Ying Hu\thanks{Univ. Rennes, CNRS, IRMAR - UMR 6625, F-35000 Rennes, France (ying.hu@univ-rennes1.fr).}~,~~
Xun Li\thanks{Department of Applied Mathematics, The Hong Kong Polytechnic University, Hong Kong
(malixun@polyu.edu.hk).}~,~~
Jiaqiang Wen\thanks{Department of Mathematics, Southern University of Science and Technology,
Shenzhen, Guangdong, 518055, China (wenjq@sustech.edu.cn).}}
\maketitle

\no\bf Abstract:
\rm
In this paper, we study the solvability of anticipated backward stochastic differential equations (BSDEs, for short) with quadratic growth for one-dimensional case and multi-dimensional case. In these BSDEs, the generator $f(\cd)$, which is of quadratic growth in $Z_\cd$, involves not only the present information of solution $(Y_\cd,Z_\cd)$ but also its future one. The existence and uniqueness of such BSDEs, under different conditions, are derived for several terminal situations, including small terminal value, bounded terminal value and unbounded terminal value.

\ms

\no\bf Key words: \rm Anticipated backward stochastic differential equation, backward stochastic
differential equation, quadratic generator, time-advanced.

\ms

\no\bf AMS subject classifications. \rm 60H10, 60H30

\section{Introduction}

Let $(\Om,\cF,\dbF,\dbP)$ be a complete filtered probability space on which a $d$-dimensional standard Brownian motion $\{W_t\;;0\les t<\i\}$ is defined, where $\dbF=\{\cF_t\}_{t\ges0}$ is the natural filtration of $W$ augmented by all the $\dbP$-null sets in $\cF$.
Let $T>0$ be a time horizon, and $K\ges0$ be a constant. Consider the following backward stochastic differential equations (BSDEs, for short) over a finite horizon $[0,T+K]$:
\bel{1.1}\left\{\ba{ll}
\ds -dY_t=f\big(t,Y_t,Z_t,Y_{t+\delta(t)},Z_{t+\zeta(t)}\big)dt-Z_tdW_t,\q~t\in[0,T]; \\
\ns\ns\ds Y_t=\xi_t,\q~Z_t=\eta_t,\q~t\in[T,T+K],
\ea\right.\ee
where $\delta(\cdot)$ and $\zeta(\cdot)$ are two deterministic $\mathbb{R}^{+}$-valued continuous functions defined on $[0,T]$, and $\xi_\cd$ and $\eta_\cd$ are some given processes.
Such an equation is called an anticipated backward stochastic differential equation (ABSDE, for short) that appears as an adjoint process when dealing with optimal control problems under delayed systems.
The unknown processes, called an adapted solution of $(\ref{1.1})$, are the pair $(Y_\cd,Z_\cd)$ of $\dbF$-adapted processes taking values in $\dbR^m\ts\dbR^{m\ts d}$. We call $\xi_\cd$ the {\it terminal value} and $f(\cd)$ the {\it generator} of the corresponding BSDE \rf{1.1}. For convenience, hereafter, by a {\it quadratic BSDE}, we mean that in BSDE \rf{1.1}, the generator $f(\cd)$ grows in $Z_\cd$ superlinearly, but no more than quadratically; if $f(\cd)$ grows in $Z_\cd$ faster than quadratic, we call it a {\it super-quadratic BSDE}.
Meanwhile, we call $\xi_\cd$ the {\it small terminal value},
if there exists some constant $\e>0$ such that $\|\xi\|_\i\les \e$; we call $\xi_\cd$ the {\it bounded terminal value},
if $\xi_\cd\in L^{\i}_\dbF(T,T+K;\dbR^m)$; and we call $\xi_\cd$ the {\it unbounded terminal value}, if
$\xi_\cd\in S_\dbF^2(T,T+K;\dbR^m)$ (see Section 2 for the definition of the norm and spaces).

\ms

Recently, as a natural extension of BSDEs (see below for precise description), Peng and Yang \cite{Peng-Yang-09} study the equation \rf{1.1} and establish the existence and uniqueness of adapted solutions under the condition that $f(\cd)$ is uniformly Lipschitz in the last four arguments. Unfortunately, the comparison with Lipschitz generator fails to hold for the general ABSDE \rf{1.1} (see the counter Example 5.3 of \cite{Peng-Yang-09}). Since then, by regarding ABSDE as a new duality type of stochastic differential delay equations (SDDEs, for short), Chen and Wu \cite{Chen-Wu-10} derive the maximum principle for stochastic optimal control problem with delay. Along this line of research, Huang and Shi \cite{Huang-Shi-12} develop the maximum principle for optimal control of fully coupled forward-backward stochastic differential equations with delay. Chen, Wu and Yu \cite{Chen-Wu-Yu-12} discuss a delayed stochastic linear-quadratic (LQ, for short) control problem and show that the Riccati equation of this problem is a quadratic anticipated (ordinary) differential equation, which is a special form of quadratic ABSDE \rf{1.1}.
Recently, Sun, Xiong and Yong \cite{Sun-Xiong-Yong-18} consider the stochastic LQ optimal control problems with random coefficients using the stochastic Riccati equation, which is, in fact, the quadratic BSDEs. In theory, Wu, Wang and Ren \cite{Wu-Wang-Ren-12} establish the existence and uniqueness of ABSDE \rf{1.1} with non-Lipschitz coefficients,
 Wen and Shi \cite{Wen-Shi-17} and Douissi, Wen and Shi \cite{Douissi-Wen-Shi19} analyze the solvability of ABSDE driven by a fractional Brownian motion and its applications in optimal control problem.

\ms

Now, let us recall the following classical BSDE:
\bel{BSDE1}\left\{\2n\ba{ll}
\ds -dY_t =f(t,Y_t,Z_t)dt-Z_tdW_t,\q~t\in[0,T];\\
\ns\ns\ds Y_T=\xi_T.\ea\right.\ee
When $(Y_\cd,Z_\cd)\mapsto f(\cd,Y_\cd,Z_\cd)$ is linear, such an equation is initially formulated by Bismut \cite{Bismut-76} in the context of maximum principle for stochastic optimal controls. Pardoux and Peng \cite{Pardoux-Peng-90} firstly investigate the general nonlinear case of \rf{BSDE1}. Since then, BSDEs attract many researchers' interest. It stimulates some significant developments in many fields, such as partial differential equation (see Pardoux and Peng \cite{Pardoux-Peng-92}), stochastic optimal control (see Yong and Zhou \cite{Yong-Zhou-99}) and mathematical finance (see El Karoui, Peng and Quenez \cite{Karoui-Peng-Quenez-97}), to mention a few. Meanwhile, many efforts have been made to relax the assumptions on the generator $f(\cd)$ of BSDE \rf{BSDE1} for the existence and/or uniqueness of adapted solutions. For example, Lepeltier and San Martin \cite{Lepeltier-Martin-97} prove the existence of adapted solutions of one-dimensional BSDE when the generator $f(\cd)$ is continuous and of linear growth in $(Y_\cd,Z_\cd)$ (without Lipschitz condition). Kobylanski \cite{Kobylanski00} establishes the well-posedness of one-dimensional BSDE \rf{BSDE1} with
$f(\cd)$ growing quadratically in $Z_\cd$ and with bounded terminal condition. Recently, Tevzadze \cite{Tevzadze-08} revisits the existence and uniqueness of quadratic BSDEs by means of a fixed point argument. Briand and Hu \cite{Briand-Hu-06,Briand-Hu-08} prove the existence and uniqueness of quadratic BSDEs with unbounded terminal value. Delbaen, Hu and Bao \cite{Delbaen-Hu-Bao11} further study super-quadratic BSDEs.
Some other recent developments of quadratic BSDEs can be found in Bahlali, Eddahbi and Ouknine \cite{Bahali-Eddahbi-Ouknine-17}, Barrieu and El Karoui \cite{Briand-Karoui-13}, Cheridito and Nam \cite{Cheridito-Nam-14,Cheridito-Nam-15}, Hibon, Hu and Tang \cite{Hibon-Hu-Tang-17}, Hu and Tang \cite{Hu-Tang-16}, Richou \cite{Richou-12}, Wen and Yong \cite{Wen-Yong-18},
 Xing and Zitkovic \cite{Xing-Zitkovic}, and references cited therein.

\ms

As an important development of BSDEs, the ABSDEs with quadratic growth, also called the quadratic ABSDEs, have important applications in stochastic optimal control problems, especially in delayed stochastic LQ optimal control problems with random coefficients (see \cite{Chen-Wu-Yu-12,Sun-Xiong-Yong-18}). To our best knowledge, only Fujii and Takahashi \cite{Fujii-Takahashi-18} consider quadratic ABSDEs, and study the existence and uniqueness of a class of one-dimensional quadratic ABSDE when the generator $f(\cd)$ is independent of the anticipated term $Z_{\cd+\zeta(\cd)}$. However, for the general ABSDE \rf{1.1} with quadratic growth, there does not exist any fundamental result yet.
In this paper, we focus on this problem and study the solvability of ABSDE \rf{1.1} with quadratic growth among one-dimensional case and multi-dimensional case. Under different conditions, with different methods, we establish the solvability of quadratic ABSDEs with small terminal value, bounded terminal value and unbounded terminal value, respectively. To tackle the difficulty of lack of comparison principle, we use the John-Nirenberg inequality for BMO-martingale to obtain the solvability of quadratic ABSDE.

\ms

Firstly, for multi-dimensional ABSDE \rf{1.1}, we discuss the case when $f(\cd)$ is of quadratic growth in the last four arguments $(Y_\cd,Z_\cd,Y_{\cd+\d(\cd)},Z_{\cd+\zeta(\cd)})$  and $\xi_\cd$ has small terminal value. In this case, we obtain the existence and uniqueness of adapted solutions.
Secondly, for one-dimensional ABSDE \rf{1.1}, we study the case when $f(\cd)$ is of quadratic growth in $Z_\cd$ and superlinear growth in $Z_{\cd+\zeta(\cd)}$, and $\xi_\cd$ is a bounded random variable. For such a case, we get the existence and uniqueness of local adapted solutions. In addition, if $f(\cd)$ is bounded in $Z_{\cd+\zeta(\cd)}$, we derive the existence and uniqueness of global adapted solutions. Thirdly, when $\xi_\cd$ is a unbounded random variable, we study the solvability of one-dimensional ABSDE \rf{1.1} with quadratic growth. The case when the generator $f(\cd)$ is of bounded growth in $(Y_\cd,Y_{\cd+\d(\cd)},Z_{\cd+\zeta(\cd)})$ and of quadratic growth in $Z_\cd$ is considered, and the existence and uniqueness of adapted solutions for this situation obtained. Compared with Fujii and Takahashi \cite{Fujii-Takahashi-18}, conclusions shed light on are the fundamental difference between theirs and ours.

\ms

The rest of the paper is organized as follows. In Section 2, we present some preliminaries and some spaces which will be used in the following sections. In Section 3, the existence and uniqueness of quadratic ABSDEs with small terminal value are established. The solvability of quadratic ABSDEs with bounded terminal value is obtained in Section 4, and we get the existence and uniqueness of quadratic ABSDEs with unbounded terminal value in Section 5.

\section{Preliminaries}

Throughout this paper, and recall from the previous section, let $(\Om,\cF,\dbF,\dbP)$ be a complete filtered probability space on which a $d$-dimensional standard Brownian motion $\{W_t\;;0\les t<\i\}$ is defined, where $\dbF=\{\cF_t\}_{t\ges0}$ is the natural filtration of $W_\cd$ augmented by all the $\dbP$-null sets in $\cF$.
The notion $\dbR^{m\ts d}$ denotes the space of the $m\ts d$-matrix $C$ with Euclidean norm $|C|=\sqrt{tr(CC^*)}$.
Next, for any $t\in[0,T]$ and Euclidean space $\dbH$, we introduce the following spaces:
$$\ba{ll}
\ns\ds L^2_{\cF_t}(\Om;\dbH)=\Big\{\theta:\Om\to\dbH\bigm|\th\hb{ is $\cF_t$-measurable, }\|\th\|_2\deq\big(\dbE|\th|^2\big)^{1\over2}<\i\Big\},\\
\ea$$
$$\ba{ll}
\ns\ds L^\i_{\cF_t}(\Om;\dbH)=\Big\{\th:\Om\to\dbH\bigm|\th\hb{ is $\cF_t$-measurable, }
\|\th\|_{\i}\triangleq\esssup_{\o\in\Om}|\th(\omega)|<\i\Big\},\\
\ns\ds L_\dbF^2(t,T;\dbH)\1n=\1n\Big\{X:[t,T]\1n\times\1n\Om\to\dbH\bigm|X_\cd\hb{ is
$\dbF$-progressively measurable, }\\
\ns\ds\qq\qq\qq\qq\qq\qq\q~
\|X_\cd\|_{L_\dbF^2(t,T)}\deq\1n\(\dbE\int^T_t\1n|X_s|^2ds\)^{1\over2}\1n<\2n\i\Big\},\\
\ns\ds L_\dbF^\infty(t,T;\dbH)=\Big\{X:[t,T]\times\Om\to\dbH\bigm|X_\cd\hb{ is
$\dbF$-progressively measurable, }\\
\ns\ds\qq\qq\qq\qq\qq\qq\q~
\|X_\cd\|_{L_\dbF^\infty(t,T)}\deq\esssup_{(s,\o)\in[t,T]\times\Om}|X_s(\o)|<\i\Big\},\\
\ns\ds S_\dbF^2(t,T;\dbH)=\Big\{X:[t,T]\times\Om\to\dbH\bigm|X_\cd\hb{ is
$\dbF$-adapted, continuous, }\\
\ns\ds\qq\qq\qq\qq\qq\qq\q~
\|X_\cd\|_{S_\dbF^2(t,T)}\deq\Big\{\dbE\(\sup_{s\in[t,T]}|X_s|^2\)\Big\}^{\frac{1}{2}}<\i\Big\}.\ea$$
Let $M=(M_t,\cF_t)$ be a uniformly integrable martingale with $M_0=0$, and for $p\in[1,\i)$, we set
$$\|M\|_{BMO_p(\dbP)}\deq
\sup_\t\bigg\|\dbE_\t\Big[\big(\langle M\rangle_{\t}^{\i}\big)^{\frac{p}{2}}\Big]^{\frac{1}{p}}\bigg\|_{\i},$$
where the supremum is taken over all $\dbF$-stopping times $\t$, and $\dbE_\t$ is the conditional expectation given $\cF_{\t}$. The class $\big\{M: \|M\|_{BMO_p(\dbP)}<\i\big\}$ is denoted by $BMO_p(\dbP)$. Observe that $\| \cd \|_{BMO_p}$ is a norm on this space and $BMO_p(\dbP)$ is a Banach space. In the sequel, we denote $BMO(\dbP)$ the space of $BMO_2(\dbP)$ for simplicity. Next, for any $Z_\cd\in L^2_\dbF(0,T;\dbH)$, by Burkholder-Davis-Gundy's inequalities, one has
$$\ba{ll}
\ns\ds c_2\dbE_\t\[\(\int_\t^T|Z_s|^2ds\)\]\les\dbE_\t\[\sup_{t\in[\t,T]}
\Big|\int_\t^tZ_sdW_s\Big|^2\]\les C_2\dbE_\t\[\(\int_\t^T|Z_s|^2ds\)\],\ea$$
for some constants $c_2,C_2>0$. Thus,
$$\ba{ll}
\ns\ds c_2\sup_{\t\in\sT[t,T]}\Big\|\dbE_\t\[\(\int_\t^T|Z_s|^2ds\)\]\Big\|_\i
\les\sup_{\t\in\sT[t,T]}\Big\|\dbE_\t\[\sup_{t\in[\t,T]}\Big|\int_\t^tZ_sdW_s\Big|^2\]\Big\|_\infty\\
\ns\ds\qq\qq\qq\qq\qq\qq\qq\q \les C_2\sup_{\t\in\sT[t,T]}\Big\|\dbE_\t\[\(\int_\t^T|Z_s|^2ds\)\]\Big\|_\i,\ea$$
where $\sT[t,T]$ denotes the set of all $\dbF$-stopping times $\t$ valued in $[t,T]$. Note that the above could be infinite. Therefore, we introduce the following:
$$\cZ^2[t,T]=\Big\{Z_\cd\in L^2_\dbF(t,T;\dbH)\Bigm|\|Z\|_{\cZ^2[t,T]}\equiv\sup_{\t\in\sT[t,T]}\Big\|
\dbE_\t\[\int_\t^T|Z_s|^2ds\] \Big\|_\i^{1\over2}<\i\Big\}.$$
Recall that for $Z_\cd\in\cZ^2[0,T]$, the process $s\mapsto\int_0^sZ_rdW_r$ (denoted by $Z\cd W$), $s\in[0,T]$, is a {\it BMO-martingale}. Moreover, note that on $[0,T]$, $Z\cd W$ belongs to $BMO(\dbP)$ if and only if $Z\in \cZ^2[0,T]$, that is,
$$\|Z\cd W\|^2_{BMO(\dbP)}\equiv \|Z\|^2_{\cZ^2[0,T]}.$$
\begin{definition}\label{8.11.21.0}
A pair $(Y_\cd,Z_\cd)\in S^2_{\dbF}(0,T+K;\dbR^m)\times
L^2_\dbF(0,T+K;\dbR^{m\ts d})$ is called an {\it adapted solution} of BSDE \rf{1.1}, if $\dbP$-almost surely, it satisfies \rf{1.1}. In addition, if $(Y_\cd,Z_\cd)\in L^\infty_\dbF(0,T+K;\dbR^m)\times\cZ^2[0,T+K]$, it is called a bounded adapted solution.
\end{definition}

Now, we recall the following three propositions, which come from Hu and Tang \cite{Hu-Tang-16} with some minor modification. The first one is the existence, uniqueness and a priori estimate for one dimensional quadratic BSDE, and the second and third ones are interesting results concerning BMO-martingale, which play important roles in our subsequent arguments.
\begin{proposition}\label{2.1.13} \sl
Let $f:\Om\ts[0,T]\ts\dbR\ts\dbR^d\ra\dbR$ be an $\mathcal{F}_t$-adapted scalar-valued generator. Moreover, there exist constants $C>0$, $\a\in(0,1)$, $\b>0$ and $\g>0$ such that for any $t\in[0,T],y,\bar y\in\dbR,z,\bar z\in\dbR^d$, we have
$$\ba{ll}
\ds |f(t,y,z)|\les|g_t|+\b|y|+\frac{\g}{2}|z|^2,\\
\ns\ds |f(t,y,z)-f(t,\bar y,\bar z)|\les C\big(|y-\bar y|+(1+|z|+|\bar z|)|z-\bar z|\big),\\
\ea$$
where $g:\Om\ts[0,T]\ra\dbR$ is $\cF_t$-adapted and $|g_t|\les |H_t|^{1+\a}$ such that the stochastic integral $H\cdot W$ is a BMO-martingale.
 Then, for any bounded random variable $\xi\in L^{\i}_{\cF_T}(\Om;\dbR)$, the following BSDE
$$Y_t=\xi+\int_t^Tf(s,Y_s,Z_s)ds-\int_t^TZ_sdW_s,\q~ t\in[0,T],$$
has a unique solution $(Y_\cd,Z_\cd)$ such that $Y_\cd$ is bounded and $Z\cd W$ is a BMO-martingale. Moreover, the following estimate holds,
$$e^{\g|Y_t|}\les\dbE_t\[e^{\g e^{\b(T-t)}|\xi|+\g\int_t^T|g_s|e^{\b(s-t)}ds}\],\q~ \forall t\in[0,T].$$

\end{proposition}

\begin{proposition}\label{2.1.11.0} \sl
For any $p\in[1,\i)$, there is a generic constant $L_p>0$ such that for any uniformly integrable martingale $M$,
$$\|M\|_{BMO_p(\dbP)}\les L_p\|M\|_{BMO(\dbP)}.$$
\end{proposition}

\begin{proposition}\label{2.1.11} \sl
For $\widetilde{K}>0$ and any one-dimensional BMO-martingale $N$ such that $\|N\|_{BMO(\dbP)}\les \widetilde{K}$.
There are constants $c_1>0$ and  $c_2>0$ depending only on $\widetilde{K}$ such that for any
BMO-martingale $M$, we have
\bel{2.1.12}c_1\|M\|_{BMO(\dbP)}\les\|\widetilde{M}\|_{BMO(\widetilde{\dbP})}\les c_2\|M\|_{BMO(\dbP)},\ee
where $\widetilde{M}\triangleq M-\langle M,N\rangle$ and $d\widetilde{\dbP}\triangleq\mathcal{E}(N)|_0^{\i}d\dbP$.
\end{proposition}


\section{Multi-dimentional Case: Small Terminal Value}

In this section, we study multi-dimensional ABSDEs with quadratic growth and small terminal value.
As showed in the theory of ordinary differential equations (ODEs, for short), the equations may not have global solutions if its generator $f(\cd)$ is super-linear with respect to $Y_\cd$. However, it should be pointed out that, in case of small bounded value, the generator could be of quadratic growth with respect to $Y_\cd$ and $Z_\cd$. In detail, let us consider the following BSDE,
\bel{3.1}\left\{\ba{ll}
\ds -dY_t=f(t,Y_t,Z_t,Y_{t+\delta(t)},Z_{t+\zeta(t)}) dt - Z_{t} dW_{t},\q~t\in [0,T]; \\
\ns\ds Y_t=\xi_t,\q~Z_{t}=\eta_t,\q~t\in[T,T+K],
\ea\right.\ee
where $\delta(\cdot)$ and $\zeta(\cdot)$ are deterministic $\mathbb{R}^{+}$-valued continuous functions defined on $[0,T]$ satisfying the following two items:
\begin{enumerate}[~~\,\rm(i)]
\item [(i)] There exists a constant $K\ges 0$ such that
    \begin{equation}\label{14.55}
      t + \delta(t) \les T+K; \q~ t + \zeta(t) \les T+K, \q~\forall t\in[0,T].
    \end{equation}
  \item [(ii)] There exists a constant $L \ges 0$ such that for all nonnegative and integrable $h_\cd$,
   \begin{equation}\label{14.56}
     \int_t^T h_{s + \delta(s)} ds \les L \int_t^{T+K} h_s ds; \q~
     \int_t^T h_{s + \zeta(s)} ds \les L \int_t^{T+K} h_s ds, \q~ \forall t\in[0,T].
   \end{equation}
\end{enumerate}
Assume that for all $s\in[0,T]$, $f(s,\o,y,z,\xi,\eta):\Om\ts\dbR^m\ts\dbR^{m\ts d}\ts L^2_{\cF_r}(\Om;\dbR^m)\ts
L^2_{\cF_{\bar{r}}}(\Om;\dbR^{m\ts d})\ra L^2_{\cF_s}(\Om;\dbR^m)$,
where $r,\bar{r}\in[s,T+K]$, and $f(\cdot)$ satisfies the following condition:
\bas{3.2} \rm
Let $C$ be a positive constant. For all
$s\in[0,T]$, $y,\bar{y}\in\dbR^m,z,\bar{z}\in\dbR^{m\ts d}$,
$\xi_{\cd},\bar{\xi}_{\cd}\in L^2_{\dbF}(s,T+K;\dbR^m),\eta_{\cd},\bar{\eta}_{\cd}\in L^2_{\dbF}(s,T+K;\dbR^{m\ts d})$, $r,\bar{r}\in[s,T+K]$, we have $f(s,0,0,0,0)$ is bounded, and
$$\ba{ll}
\ds\big|f(s,y,z,\xi_r,\eta_{\bar{r}})-f(s,\bar{y},\bar{z},\bar{\xi}_r,\bar{\eta}_{\bar{r}})\big|
\les C\Big(|y|+|\bar{y}|+|z|+|\bar{z}|+\dbE_s\big[|\xi_r|+|\bar{\xi}_r|+|\eta_{\bar{r}}|+|\bar{\eta}_{\bar{r}}|\big]\Big)\\
\ns\ds\qq\qq\qq\qq\qq\qq\qq\qq \
\cd\Big(|y-\bar{y}|+|z-\bar{z}|+\dbE_s\big[|\xi_r-\bar{\xi}_r|+|\eta_{\bar{r}}-\bar{\eta}_{\bar{r}}|\big]\Big).
\ea$$
\eas

\bas{3.3} \rm
The given terminal conditions $\xi_\cd\in L_\dbF^\infty(T,T+K;\dbR^m)$ and $\eta_\cd\in \cZ^2[T,T+K]$.
\eas

\begin{example} \rm
Assumption \ref{3.2} implies that the generator $f(\cd)$ could be of quadratic growth with respect to the last four arguments. The following generator
$$\ba{ll}
\ds\qq\qq\qq\qq\q f(s,y,z,\xi_r,\eta_{\bar r})=y^2+z^2+\dbE_s\big[\xi^2_r+\eta^2_{\bar r}\big],\\
%
%
\ns\ds \forall s\in[0,T], y\in\dbR^m,z\in\dbR^{m\ts d},
\xi_{\cd}\in L^2_{\dbF}(s,T+K;\dbR^m),\eta_{\cd}\in L^2_{\dbF}(s,T+K;\dbR^{m\ts d}), r,\bar{r}\in[s,T+K]
\ea$$
satisfies such an assumption.
\end{example}

Next, we state and prove the main result of this section, which establishes the existence and uniqueness of BSDE \rf{3.1} with quadratic growth and small terminal value.

\bt{3.4}\sl Let  Assumptions \ref{3.2} and \ref{3.3} hold. There exists a small $\rho>0$ such that when
\bel{19.3.21.3}\|\xi_\cd\|^2_{L_\dbF^\infty(T,T+K)}+\|\eta_\cd\|^2_{\cZ^2[T,T+K]}
+\bigg\|\int_0^T|f(t,0,0,0,0)|dt\bigg\|^2_{\i}\les \rho^2,\ee
BSDE \rf{3.1} admits a unique adapted solution $(Y_\cd,Z_\cd)$ in $\cB_\rho$, where
$$\cB_\rho\deq\Big\{(Y_\cd,Z_\cd)\in L_\dbF^\infty(0,T+K;\dbR^m)\times \cZ^2[0,T+K]\Bigm|
\|Y_\cd\|^2_{L_\dbF^\infty(0,T+K)}+\|Z_\cd\|_{\cZ^2[0,T+K]}^2\les \rho^2\Big\}.$$
\et

\it Proof. \rm
The proof is divided into two steps.

\ms

\textbf{Step 1}.  We firstly consider the existence and uniqueness of the following BSDE
\bel{3.5.0}\left\{\ba{ll}
\ds -dY_t=\Big(f(t,Y_t,Z_t,Y_{t+\delta(t)},Z_{t+\zeta(t)})-f(t,0,0,0,0)\Big)dt-Z_tdW_t,\q~ t\in[0,T]; \\
\ns\ds Y_t=\xi_t, \q~  Z_{t}=\eta_t, \q~  t\in[T,T+K].
\ea\right.\ee
In order to solve the above equation, for every $(y_\cd,z_\cd)\in L_\dbF^\infty(0,T+K;\dbR^m)\times \cZ^2[0,T+K]$,
we define the mapping $(Y_\cd,Z_\cd)=\Gamma(y_\cd,z_\cd)$ by
\bel{3.5}\left\{\ba{ll}
\ds -dY_t=\Big(f(t,y_t,z_t,y_{t+\delta(t)},z_{t+\zeta(t)})-f(t,0,0,0,0)\Big)dt-Z_tdW_t,\q~ t\in[0,T]; \\
\ns\ds Y_t=\xi_t, \q~  Z_{t}=\eta_t, \q~  t\in[T,T+K].
\ea\right.\ee
Note that since $Y_{t}=\xi_t$ and $Z_{t}=\eta_{t}$ are given when $t\in[T,T+K]$,
we essentially need to prove the estimate on $[0,T]$. For \rf{3.5},
using It\^{o}'s formula to $|Y_\cd|^2$ on $[t,T]$, we obtain
$$\ba{ll}
\ds|Y_t|^2+\int_t^T|Z_r|^2dr=|\xi_T|^2
+\int_t^T2Y_r\cdot\Big(f(r,y_r,z_r,y_{r+\delta(r)},z_{r+\zeta(r)})-f(r,0,0,0,0)\Big)dr\\
\ns\ns\ds\qq\qq\qq\qq\q\ -2\int_t^TY_r\cdot Z_rdW_r.
\ea$$
Taking the conditional expectation and using the
inequality $2ab\les\frac{1}{2}a^2+2b^2$, we get
\bel{3.6}\ba{ll}
\ds |Y_t|^2+\dbE_{t}\int_t^T|Z_r|^2dr\\
\ns\ds\les \|\xi_\cd\|_{L_\dbF^\infty(T,T+K)}^2
+2\|Y_\cd\|_{L_\dbF^\infty(0,T)}
\Bigg(\dbE_{t}\int_t^T|f(r,y_r,z_r,y_{r+\delta(r)},z_{r+\zeta(r)})-f(r,0,0,0,0)|dr\Bigg)\\
\ns\ds\les\|\xi_\cd\|_{L_\dbF^\infty(T,T+K)}^2+\frac{1}{2}\|Y_\cd\|_{L_\dbF^\infty(0,T)}^2
+2\Bigg(\dbE_{t}\int_t^T|f(r,y_r,z_r,y_{r+\delta(r)},z_{r+\zeta(r)})-f(r,0,0,0,0)|dr\Bigg)^{2}.
\ea\ee
Since $\delta(\cd)$ and $\zeta(\cd)$ satisfy (i) and (ii), it follows from Assumption \ref{3.2} and Jensen's inequality that the last term of \rf{3.6} naturally reduces to
\bel{3.7}\ba{ll}
\ds \dbE_{t}\int_t^T|f(r,y_r,z_r,y_{r+\delta(r)},z_{r+\zeta(r)})-f(r,0,0,0,0)|dr\\
\ns\ds\les C\dbE_{t}\int_t^T\Big(|y_r|+|z_r|+\dbE_r\big[|y_{r+\delta(r)}|+|z_{r+\zeta(r)}|\big]\Big)^2dr\\
\ns\ds\les 4C\dbE_{t}\int_t^T\Big(|y_r|^2+|z_r|^2+\dbE_r\big[|y_{r+\delta(r)}|^2+|z_{r+\zeta(r)}|^2\big]\Big)dr\\
\ns\ds\les 4C(1+L)\dbE_{t}\int_t^{T+K}\big(|y_r|^2+|z_r|^2\big)dr,
\ea\ee
here $\dbE_{t}[\dbE_{r}[\ \cdot\ ]]=\dbE_{t}[\ \cdot\ ]$.
Hence, combining \rf{3.6} and \rf{3.7}, we can obtain
$$\ba{ll}
\ds \frac{1}{2}\|Y_\cd\|^2_{L_\dbF^\infty(0,T)}+\|Z_\cd\|_{\cZ^2[0,T]}^{2}\\
\ns\ds\les
\|\xi_\cd\|_{L_\dbF^\infty(T,T+K)}^2+2\esssup_{(t,\o)\in[0,T]\times\Om}
\bigg(4C(1+L)\dbE_{t}\int_t^{T+K}\big(|y_r|^2+|z_r|^2\big)dr\bigg)^{2}\\
\ns\ds\les\|\xi_\cd\|_{L_\dbF^\infty(T,T+K)}^2+64C^2(1+L)^2
\bigg((T+K)^2\|y_\cd\|_{L_\dbF^\infty(0,T+K)}^{4}+\|z_\cd\|_{\cZ^2[0,T+K]}^{4}\bigg).
\ea$$
Again, note that $Y_{t}=\xi_t$ and $Z_{t}=\eta_{t}$ when $t\in[T,T+K]$.
Then, it follows from the elementary inequality $a^2+b^2\les(|a|+|b|)^2$ that
$$\ba{ll}
\ds \|Y_\cd\|^{2}_{L_\dbF^\infty(0,T+K)}+\|Z_\cd\|_{\cZ^2[0,T+K]}^{2}\\
\ns\ds \les4\Big(\|\xi_\cd\|_{L_\dbF^\infty(T,T+K)}^2+\|\eta_\cd\|_{\cZ^2[T,T+K]}^{2}\Big)
+\b^2\Big(\|y_\cd\|_{L_\dbF^\infty(0,T+K)}^2+\|z_\cd\|_{\cZ^2[0,T+K]}^{2}\Big)^{2},
\ea$$
where $\b\deq 16C(1+L)\sqrt{(T+K)^2+1}$. Now, we can pick $R$ such that
$$4\Big(\|\xi_\cd\|_{L_\dbF^\infty(T,T+K)}^2+\|\eta_\cd\|_{\cZ^2[T,T+K]}^{2}\Big) + \b^2R^4\les R^2.$$
This inequality is solvable if and only if
\bel{19.3.21.1}\|\xi_\cd\|_{L_\dbF^\infty(T,T+K)}^2+\|\eta_\cd\|_{\cZ^2[T,T+K]}^{2}\les \frac{1}{16\b^2}.\ee
For example, we can take
$$R=\sqrt{8\Big(\|\xi_\cd\|_{L_\dbF^\infty(T,T+K)}^2+\|\eta_\cd\|_{\cZ^2[T,T+K]}^{2}\Big)}$$
in order to satisfy this quadratic inequality. Therefore the ball
$$\cB_R\deq\Big\{(Y_\cd,Z_\cd)\in L_\dbF^\infty(0,T+K;\dbR^m)\times \cZ^2[0,T+K]\Bigm|
\|Y_\cd\|^2_{L_\dbF^\infty(0,T+K)}+\|Z_\cd\|_{\cZ^2[0,T+K]}^2\les R^2\Big\}$$
is such that $\G(\cB_R)\subset\cB_R$.

\ms

\textbf{Step 2}. We prove that the mapping $\G$ is a contraction on $\cB_R$.

\ms

For every $(y_\cd,z_\cd)$, $(\by_\cd,\bz_\cd)\in\cB_R$, let $(Y_\cd,Z_\cd)=\G(y_\cd,z_\cd)$ and
$(\bY_\cd,\bZ_\cd)=\G(\by_\cd,\bz_\cd)$. For simplicity of presentation, denote
$$\hy_\cd=y_\cd-\by_\cd,\q~\hz_\cd=z_\cd-\bz_\cd,\q~\hY_\cd=Y_\cd-\bY_\cd,\q~\hY_\cd=Y_\cd-\bY_\cd.$$
Similar to the above discussion, and note that $\hY_t=0$ and $\hZ_t=0$ when $t\in[T,T+K]$, we get
$$\ba{ll}
\ds\frac{1}{2}\|\hY_\cd\|^{2}_{L_\dbF^\infty(0,T+K)}+\|\hZ_\cd\|_{\cZ^2[0,T+K]}^{2}
=\frac{1}{2}\|\hY_\cd\|^{2}_{L_\dbF^\infty(0,T)}+\|\hZ_\cd\|_{\cZ^2[0,T]}^{2}\\
\ns\ds\les2\esssup_{(t,\o)\in[0,T]\times\Om}\bigg[\dbE_{t}\int_t^T
\Big|f(r,y_r,z_r,y_{r+\delta(r)},z_{r+\zeta(r)})
-f(r,\by_r,\bz_r,\by_{r+\delta(r)},\bz_{r+\zeta(r)})\Big|dr\bigg]^{2}\\
\ns\ds\les 2C^2\esssup_{(t,\o)\in[0,T]\times\Om}\bigg[\dbE_{t}\int_t^T
\Big(|y_r|+|z_r|+|\by_r|+|\bz_r|+\dbE_r\big[|y_{r+\delta(r)}|+|z_{r+\zeta(r)}|
+|\by_{r+\delta(r)}|+|\bz_{r+\zeta(r)}|\big]\Big)\\
\ns\ds\qq\qq\qq\qq\qq\q\
\cd\Big(|\hy_r|+|\hz_r|+\dbE_r\big[|\hy_{r+\delta(r)}|+|\hz_{r+\zeta(r)}|\big]\Big)dr\bigg]^{2}\\
\ns\ds\les 2C^2\esssup_{(t,\o)\in[0,T]\times\Om}\bigg[\dbE_{t}\int_t^T
\Big(|y_r|+|z_r|+|\by_r|+|\bz_r|+\dbE_r\big[|y_{r+\delta(r)}|+|z_{r+\zeta(r)}|
+|\by_{r+\delta(r)}|+|\bz_{r+\zeta(r)}|\big]\Big)^2dr\\
\ns\ds\qq\qq\qq\qq~
\cd\dbE_{t}\int_t^T\Big(|\hy_r|+|\hz_r|+\dbE_r\big[|\hy_{r+\delta(r)}|+|\hz_{r+\zeta(r)}|\big]\Big)^2dr\bigg]\\
\ns\ds\les 64C^2(1+L)^2\esssup_{(t,\o)\in[0,T]\times\Om}\bigg[\dbE_{t}\int_t^{T+K}
\Big(|y_r|^2+|z_r|^2+|\by_r|^2+|\bz_r|^2\Big)dr\cd\dbE_{t}\int_t^{T+K}\Big(|\hy_r|^2+|\hz_r|^2\Big)dr\bigg]\\
\ns\ns\ds\les
64C^2(1+L)^2[(T+K)^2+1]
\Big[\|y_\cd\|^{2}_{L_\dbF^\infty(0,T+K)}+\|z_\cd\|^{2}_{\cZ^2[0,T+K]}+\|\by_\cd\|^{2}_{L_\dbF^\infty(0,T+K)}
+\|\bz_\cd\|^{2}_{\cZ^2[0,T+K]}\Big]\\
\ns\ns\ds\q\ \times\Big(\|\hy_\cd\|^{2}_{L_\dbF^\infty(0,T+K)}+\|\hz_\cd\|^{2}_{\cZ^2[0,T+K]}\Big).
\ea$$
Note that
$$\|y_\cd\|^{2}_{L_\dbF^\infty(0,T+K)}+\|z_\cd\|^{2}_{\cZ^2[0,T+K]}\les R^2,\qq
\|\by_\cd\|^{2}_{L_\dbF^\infty(0,T+K)}+\|\bz_\cd\|^{2}_{\cZ^2[0,T+K]}\les R^2,$$
we obtain
$$\|\hY_\cd\|^{2}_{L_\dbF^\infty(0,T+K)}+\|\hZ_\cd\|_{\cZ^2[0,T+K]}^{2}
\les MR^2\big(\|\hy_\cd\|^{2}_{L_\dbF^\infty(0,T+K)}+\|\hz_\cd\|^{2}_{\cZ^2[0,T+K]}\big),$$
where
$$M\deq 256C^2(1+L)^2[(T+K)^2+1].$$
Now we take
$$R=\sqrt{8\Big(\|\xi_\cd\|_{L_\dbF^\infty(T,T+K)}^2+\|\eta_\cd\|_{\cZ^2[T,T+K]}^{2}\Big)},$$
and let $R<\frac{1}{\sqrt{M}}$, which implies that
\bel{19.3.21.2}\|\xi_\cd\|_{L_\dbF^\infty(T,T+K)}^2+\|\eta_\cd\|_{\cZ^2[T,T+K]}^{2}<\frac{1}{8M},\ee
and $\G$ is a contraction on $\cB_R$. By the contraction principle, the mapping $\G$ admits a unique fixed point,
which is the solution of \rf{3.5.0}. Finally, we come back to BSDE \rf{3.1}. Note that \rf{19.3.21.2} is stronger than  \rf{19.3.21.1}. Now we define $\rho>0$ by letting
$$\rho^2\deq \frac{1}{16M}=\frac{1}{4096C^2(1+L)^2[(T+K)^2+1]},$$
then, when \rf{19.3.21.3} holds, we have that BSDE \rf{3.1} admits a unique adapted solution
$(Y_\cd,Z_\cd)\in\cB_\rho$. This completes the proof.
\endpf

\section{Scalar Case: Bounded Terminal Value}

In this section, we study the solvability of one-dimensional ABSDEs with quadratic growth and bounded terminal value.
That is in the following we shall assume that $m=1$. And for simplicity of presentation, we also let $d=1$. In particular, the case of local solution and the case of global solution are investigated respectively.

\subsection{Local Solution}

In this subsection, the existence and uniqueness of local adapted solution of BSDE \rf{1.1} are studied.
For simplicity, we rewrite it as follows:
$$\left\{\ba{ll}
\ds -dY_t=f\big(t,Y_t,Z_t,Y_{t+\delta(t)},Z_{t+\zeta(t)}\big)dt-Z_tdW_t,\q~t\in[0,T]; \\
\ns\ns\ds Y_t=\xi_t,\q~Z_t=\eta_t,\q~t\in[T,T+K],
\ea\right.$$
where, in this section, the deterministic continuous functions $\delta(\cdot)$ and $\zeta(\cdot)$ satisfy the following two items:
\begin{enumerate}[~~\,\rm(i)]
\item [(i)] There exists a constant $K\ges 0$ such that
    \begin{equation*}
      t + \delta(t) \les T+K; \q~ t + \zeta(t) \les T+K, \q~\forall t\in[0,T].
    \end{equation*}
  \item [(ii)] There exists a constant $L \ges 0$ such that for all nonnegative and integrable $h_\cd$,
   \begin{equation}\label{14.57}
     \int_t^T h_{s + \delta(s)} ds \les L \int_t^{T+K} h_s ds; \q~
     \int_t^T h_{s + \zeta(s)} ds \les L\int_t^{T+K} h_s ds, \q~ \forall t\in[0,T].
   \end{equation}
\end{enumerate}
%
%
Now, before further specifying, let us present the assumptions.
\bas{4.1.1} \rm
Assume that for all $s\in[0,T]$, $f(s,\o,y,z,\xi,\eta):\Om\ts\dbR\ts\dbR\ts L^2_{\cF_r}(\Om;\dbR)\ts
L^2_{\cF_{\bar{r}}}(\Om;\dbR)\ra L^2_{\cF_s}(\Om;\dbR)$, where $r,\bar{r}\in[s,T+K]$.
Let $C$ and $\g$ be positive constants and $\a\in[0,1)$. For all $s\in[0,T]$, $y,\bar{y},z,\bar{z}\in\dbR$,
$\xi_{\cd},\bar{\xi}_{\cd},\eta_{\cd},\bar{\eta}_{\cd}\in L^2_{\dbF}(s,T+K;\dbR)$, we have
$$\ba{ll}
\ds |f(s,y,z,\xi_r,\eta_{\bar{r}})|\les C\Big(1+|y|+\dbE_s\big[|\xi_r|+|\eta_{\bar{r}}|^{1+\a}\big]\Big)+\frac{\g}{2}|z|^2;\\
\ns\ds
|f(s,y,z,\xi_r,\eta_{\bar{r}})-f(s,\bar{y},\bar{z},\bar{\xi}_r,\bar{\eta}_{\bar{r}})|\\
\ns\ds
\les C\bigg\{|y-\bar{y}|+\dbE_s\big[|\xi_r-\bar{\xi}_r|\big]
+\big(1+\dbE_s\big[|\eta_{\bar{r}}|^{\a}+|\bar{\eta}_{\bar{r}}|^{\a}\big]\big)
\cd\dbE_s\big[|\eta_{\bar{r}}-\bar{\eta}_{\bar{r}}|\big]\\
\ns\ds\qq~
+\big(1+|z|+|\bar z|\big)|z-\bar z|\bigg\}.
\ea$$
\eas

\bas{4.1.2} \rm
The given terminal value $\xi_\cd \in L_\dbF^\infty(T,T+K;\dbR)$ and $\eta_\cd\in\cZ^2[T,T+K]$.
\eas

Note that, for BSDE \rf{1.1}, when $t\in[T,T+K]$, the values of $(Y_\cd,Z_\cd)$ are determined by the values of $(\xi_\cd,\eta_\cd)$. Now, we introduce the following space, in which the values of the elements are determined by the values of $(\xi_\cd,\eta_\cd)$ when $t\in[T,T+K]$,
$$\cA(0,T+K)\deq\{(U_\cd,V_\cd)\in L_\dbF^\infty(0,T+K;\dbR)\ts \cZ^2[0,T+K]~\big|~
U_t=\xi_t,\q V_t=\eta_t, \q t\in[T,T+K]\}.$$

\begin{example} \rm
Assumption \ref{4.1.1} implies that $f(\cd)$ is of linear growth with respect to $y$ and $\xi_\cd$, and of sub-quadratic growth with respect to $\eta_\cd$, and of quadratic growth with respect to $z$. For instance, for $\a\in[0,1)$, the following generator satisfies such an assumption:
$$\ba{ll}
\ds f(s,y,z,\xi_r,\eta_{\bar r})=1+|y|+|z|^2+\dbE_s\big[|\xi_r|+|\eta_{\bar r}|^{1+\a}\big],\\
\ns\ds \qq  \forall s\in[0,T], \q y,z\in\dbR, \q \xi_{\cd},\eta_{\cd}\in L^2_{\dbF}(s,T+K;\dbR), \q r,\bar{r}\in[s,T+K].
\ea$$
\end{example}

In this subsection, the main result is the following theorem, which concerns the solvability of local adapted solution of BSDE \rf{1.1}.

\bt{4.1.3}\sl
Under Assumptions \ref{4.1.1} and \ref{4.1.2},
there exist some positive constant $\e$ and a bounded set $\mathcal{B}_{\e}$ such that, in the interval $[T-\e,T+K]$, the equation \rf{1.1} admits a unique local adapted solution $(Y_\cd,Z_\cd)\in\cB_{\e}$.
\et

\begin{remark} \rm
The bounded set $\mathcal{B}_{\e}$ appears in Theorem \ref{4.1.3} is a product space $L_\dbF^\infty(0,T+K;\dbR)\ts \cZ^2[0,T+K]$ restricted on the time interval $[T-\e,T+K]$. See \rf{19.4.11.1} below for its detailed definition.
\end{remark}

\it The Proof of Theorem \ref{4.1.3}. \rm
Theorem \ref{4.1.3} is proved using the contraction mapping principle in three steps. In step 1, we shall construct a mapping in a Banach space, which we call it the quadratic solution mapping. In step 2, we shall show that the above constructed mapping is stable in a small ball. In step 3, we shall prove that this mapping is a contraction.

\subsubsection*{Step 1: Construction of the mapping}

For a pair of adapted process $(U_\cd,V_\cd)\in \cA(0,T+K)$, we consider the following quadratic BSDE,
\bel{4.2.1.1}\left\{\ba{ll}
\ds -dY_t=f(t,Y_t,Z_t,U_{t+\delta(t)},V_{t+\zeta(t)}) dt - Z_{t} dW_{t}, \q~ t\in [0,T]; \\
\ns\ds Y_t=\xi_t, \q~  Z_{t}=\eta_t, \q~  t\in[T,T+K].
\ea\right.\ee
Since when $t\in[T,T+K]$, the values of $(Y_\cd,Z_\cd)$ are determined by the values of $(\xi_\cd,\eta_\cd)$,
so we essentially need to investigate \rf{4.2.1.1} in $[0,T]$.
For simplicity, we rewrite \rf{4.2.1.1} in $[0,T]$ as the following integral form,
\bel{4.2.1.2}Y_t=\xi_T+\int_t^Tf(s,Y_s,Z_s,U_{s+\delta(s)},V_{s+\zeta(s)}) ds -\int_t^T Z_{s} dW_{s},\q~ t\in[0,T].\ee
Note that $\xi_T$ is a random variable, which represents the value of the process $\xi_t$ at $t=T$.
As
$$\big|f(s,y,z,U_{s+\delta(s)},V_{s+\zeta(s)})\big|
\les C\Big(1+|y|+\dbE_s[|U_{s+\delta(s)}|]+\dbE_s\big[|V_{s+\zeta(s)}|^{1+\a}\big]\Big)+\frac{\g}{2}|z|^2,$$
in view of Proposition \ref{2.1.13}, we see that BSDE \rf{4.2.1.2} admits a unique adapted solution $(Y_\cd,Z_\cd)$, where $Y_\cd$ is a bounded process and $Z\cd W$ is a BMO-martingale. Define the quadratic solution mapping $\G:(U_\cd,V_\cd)\mapsto\G(U_\cd,V_\cd)$ as follows,
$$\G(U_\cd,V_\cd)\triangleq(Y_\cd,Z_\cd),\qq \forall(U_\cd,V_\cd)\in \cA(0,T+K).$$
This is a transformation in the Banach space $L_\dbF^\infty(0,T+K;\dbR)\ts \cZ^2[0,T+K]$.

\ms
Hereafter, for simplicity of presentation, we denote
\bel{4.2.1.5}\|\xi_\cd\|_\i\deq\|\xi_\cd\|_{L_\dbF^\infty(T,T+K)},\qq
\|\eta_\cd\|_{\cZ^2}\deq\|\eta_\cd\|_{\cZ^2[T,T+K]}.\ee
Introduce some constants and a quadratic (algebraic) equation which will be adopted in the following subsections. Set
%
\bel{4.2.1.4}\left\{\ba{ll}
\ds C_\d\triangleq e^{\frac{3\g}{1-\a} e^{CT}C(T+K)+\frac{1}{L}\frac{1-\a}{2}\Big(\frac{3\g}{1-\a} CLe^{CT}\Big)^{\frac{2}{1-\a}}\Big(\frac{1+\a}{2\d}\Big)^{\frac{1+\a}{1-\a}}(T+K)},\\
\ns\ns\ds \ \b\triangleq\frac{1}{2}(1-\a)C^{\frac{2}{1-\a}}\Big(2L(1+\a)\Big)^{\frac{1+\a}{1-\a}},\\
\ns\ns\ds  \mu_1\triangleq(1-\a)\bigg(1+\frac{1-\a}{(1+\a)\g}\bigg)=1-\a+\frac{(1-\a)^2}{(1+\a)\g},\\
\ns\ns\ds  \mu_2\triangleq\frac{1}{2}(1+\a)
\bigg(1+\frac{1-\a}{(1+\a)\g}\bigg)=\frac{1+\a}{2}+\frac{1-\a}{2\g},\\
\ns\ns\ds \ \mu\triangleq(\b+C\mu_1)\g^{\frac{2}{\a-1}}+C\mu_2(1+L),\\
\ns\ns\ds\ \tilde{\mu}\deq\g^{-2}+C\mu_2LK.
\ea\right.\ee
Consider the following standard quadratic equation in variable $A$:
$$\ba{ll}
\ds \d A^2-\Big[1+4\Big(\tilde{\mu}e^{\frac{2\g}{1-\a}\|\xi_{\cd}\|_{\i}}+\frac{1}{4}\|\eta_\cd\|^2_{\cZ^2}\Big)\d\Big]A\\
\ns\ns\ds +4\Big(\tilde{\mu}e^{\frac{2\g}{1-\a}\|\xi_{\cd}\|_{\i}}+\frac{1}{4}\|\eta_\cd\|^2_{\cZ^2}\Big)
+4\mu C_{\d}e^{\big(\frac{3\g}{1-\a}e^{CT}(1+CLK)\|\xi_\cd\|_{\i}+\d\|\eta_\cd\|^2_{\cZ^2}\big)}\e=0.
\ea$$
The discriminant of the above equation is
$$\ba{ll}
\ds \Delta\deq \Big[1+4\Big(\tilde{\mu}e^{\frac{2\g}{1-\a}\|\xi_{\cd}\|_{\i}}+\frac{1}{4}\|\eta_\cd\|^2_{\cZ^2}\Big)\d\Big]^2\\
\ns\ns\ds\qq\ -4\d\Big[4\Big(\tilde{\mu}e^{\frac{2\g}{1-\a}\|\xi_{\cd}\|_{\i}}+\frac{1}{4}\|\eta_\cd\|^2_{\cZ^2}\Big)
+4\mu C_{\d}e^{\big(\frac{3\g}{1-\a}e^{CT}(1+CLK)\|\xi_\cd\|_{\i}
+\d\|\eta_\cd\|^2_{\cZ^2}\big)}\e\Big]\\
\ns\ds \ \ \
=\Big[1-4\Big(\tilde{\mu}e^{\frac{2\g}{1-\a}\|\xi_{\cd}\|_{\i}}+\frac{1}{4}\|\eta_\cd\|^2_{\cZ^2}\Big)\d\Big]^2
-16\d\mu C_{\d}e^{\big(\frac{3\g}{1-\a}e^{CT}(1+CLK)\|\xi_\cd\|_{\i}.
+\d\|\eta_\cd\|^2_{\cZ^2}\big)}\e.\\
\ea$$
Taking
\bel{4.2.11}\left\{\ba{ll}
\ds \d\triangleq\frac{1}{8}\Big(\tilde{\mu}e^{\frac{2\g}{1-\a}\|\xi_{\cd}\|_{\i}}+\frac{1}{4}\|\eta_\cd\|^2_{\cZ^2}\Big)^{-1},\qq
\varepsilon\les\min\bigg\{\frac{e^{-CT}}{3CL},
\frac{\tilde{\mu}e^{\frac{2\g}{1-\a}\|\xi_{\cd}\|_{\i}}+\frac{1}{4}\|\eta_\cd\|^2_{\cZ^2}}
{8\mu C_{\d}e^{\big(\frac{3\g}{1-\a}e^{CT}(1+CLK)\|\xi_\cd\|_{\i}.
+\d\|\eta_\cd\|^2_{\cZ^2}\big)}}
  \bigg\},\\
\ns\ns\ds A\triangleq \frac{\Big[1+4\Big(\tilde{\mu}e^{\frac{2\g}{1-\a}\|\xi_{\cd}\|_{\i}}+\frac{1}{4}\|\eta_\cd\|^2_{\cZ^2}\Big)\d\Big]-\sqrt{\Delta}}{2\d}=\frac{3-2\sqrt{\D}}{4\d}\les \frac{3}{4\d}=6\Big(\tilde{\mu}e^{\frac{2\g}{1-\a}\|\xi_{\cd}\|_{\i}}+\frac{1}{4}\|\eta_\cd\|^2_{\cZ^2}\Big),
\ea\right.\ee
we have
\bel{4.2.1.6}\ba{ll}
\ds \D\ges 0,\qq 1-\d A=\frac{1+2\sqrt{\D}}{4},\\
\ns\ds \Big(\tilde{\mu}e^{\frac{2\g}{1-\a}\|\xi_{\cd}\|_{\i}}+\frac{1}{4}\|\eta_\cd\|^2_{\cZ^2}\Big)
+\mu C_{\d}\frac{e^{\big(\frac{3\g}{1-\a}e^{CT}(1+CLK)\|\xi_\cd\|_{\i}
+\d\|\eta_\cd\|^2_{\cZ^2}\big)}}{1-\d A}\varepsilon
+\frac{1}{4}A=\frac{1}{2}A.
\ea\ee

\ms

In this section, we discuss the existence and uniqueness of quadratic BSDE \rf{1.1} over the time interval $[T-\e,T+K]$.
We shall prove Theorem \ref{4.1.3} using the fact that the quadratic solution mapping $\G$ is a contraction
on the closed convex set $\mathcal{B}_{\e}$ defined by
\bel{19.4.11.1}\ba{ll}
\ds\mathcal{B}_{\e}\triangleq\bigg\{  (U_\cd,V_\cd)\in\cA(T-\e,T+K)~\Big|~
e^{\frac{2\g}{1-\a}\|U_\cd\|_{L_\dbF^\infty(T-\e,T)}}\les C_{\d}
\frac{e^{\big(\frac{3\g}{1-\a}e^{CT}(1+CLK)\|\xi_\cd\|_{\i}
+\d\|\eta_\cd\|^2_{\cZ^2}\big)}}{1-\d A},\\
\ns\ds\qq\qq\qq\qq\qq\qq\qq\qq\q~\|V_\cd\|^2_{\cZ^2[T-\e,T]}\les A\bigg\}.
\ea\ee
Again, note that in $[T,T+K]$, the values of $(Y_\cd,Z_\cd)$ are determined by $(\xi_\cd,\eta_\cd)$ respectively, therefore we essentially need to prove that \rf{1.1} has a unique adapted solution on $[T-\e,T]$.

\subsubsection*{Step 2: Estimates of the quadratic solution mapping}

Let us prove the following assertion: $\G(\mathcal{B}_{\e})\subset\mathcal{B}_{\e}$, that is,
\bel{4.2.2.1}\G(U_\cd,V_\cd)\in\mathcal{B}_{\e},\qq \forall (U_\cd,V_\cd)\in\mathcal{B}_{\e}.\ee
In order to do this, the proof is divided into the following three steps.

\ms

$\emph{Step 2.1 Exponential transformation.}$

\ms

\ms

Define
$$\p(y)\triangleq\gamma^{-2}[e^{\g|y|}-\g|y|-1],\qq y\in\dbR.$$
Then, for $y\in\dbR$,
\bel{4.2.2.3}\p'(y)=\g^{-1}[e^{\g|y|}-1]\sgn(y),\qq
\p''(y)=e^{\g|y|},\qq
\p''(y)-\g|\p'(y)|=1.\ee
Using It\^{o}'s formula to \rf{4.2.1.2}, we have for $t\in[T-\e,T]$,
\bel{4.2.2.2}\ba{ll}
\ds \p(Y_t)+\frac{1}{2}\dbE_t\int_t^T|Z_s|^2ds\\
\ns\ds\les \p(\|\xi_\cd\|_{\i})+C\dbE_t\int_t^T|\p'(Y_s)|
\Big(2+|Y_s|+\dbE_s\big[|U_{s+\d(s)}|\big]+\dbE_s\big[|V_{s+\zeta(s)}|^{1+\a}\big]\Big)ds.
\ea\ee
Using the following inequality, together with the definition of $\b$ in \rf{4.2.1.4},
$$C|\p'(Y_s)|\dbE_s\big[|V_{s+\zeta(s)}|^{1+\a}\big]
\les\b|\p'(Y_s)|^{\frac{2}{1-\a}}+\frac{1}{4L}\dbE_s\big[|V_{s+\zeta(s)}|^2\big],$$
we obtain
\bel{4.2.2.4}\ba{ll}
\ds \p(Y_t)+\frac{1}{2}\dbE_t\int_t^T|Z_s|^2ds\\
\ns\ds\les \p(\|\xi_\cd\|_{\i})+C\dbE_t\int_t^T|\p'(Y_s)|\big(2+|Y_s|+\dbE_s[|U_{s+\d(s)}|]\big)ds\\
\ns\ds \ \ \ +\b\dbE_t\int_t^T|\p'(Y_s)|^{\frac{2}{1-\a}}ds+\frac{1}{4L}\dbE_t\int_t^T\dbE_s[|V_{s+\zeta(s)}|^2]ds\\
\ns\ds\les \p(\|\xi_\cd\|_{\i})+C\dbE_t\int_t^T|\p'(Y_s)|\big((1+|Y_s|)+(1+\dbE_s[|U_{s+\d(s)}|])\big)ds\\
\ns\ds \ \ \ +\b\dbE_t\int_t^T|\p'(Y_s)|^{\frac{2}{1-\a}}ds+\frac{1}{4L}\dbE_t\int_t^T|V_{s+\zeta(s)}|^2ds.
\ea\ee
Based on the following inequality for $x>0$,
$$1+x\les \bigg(1+\frac{1-\a}{\g(1+\a)}\bigg)e^{\frac{\g(1+\a)}{1-\a}x},$$
we get
\bel{4.2.2.5}\ba{ll}
\ds C\dbE_t\int_t^T|\p'(Y_s)|\Big((1+|Y_s|)+(1+\dbE_s[|U_{s+\d(s)}|])\Big)ds\\
\ns\ds\les C\dbE_t\int_t^T|\p'(Y_s)|\bigg(1+\frac{1-\a}{\g(1+\a)}\bigg)
\bigg(e^{\frac{\g(1+\a)}{1-\a}|Y_s|}+e^{\frac{\g(1+\a)}{1-\a}\dbE_s[|U_{s+\d(s)}|]}\bigg)ds.
\ea\ee
It follows from Young's inequality that we have
\bel{4.2.2.6}\ba{ll}
\ds |\p'(Y_s)|\bigg(e^{\frac{\g(1+\a)}{1-\a}|Y_s|}+e^{\frac{\g(1+\a)}{1-\a}\dbE_s[|U_{s+\d(s)}|]}\bigg)\\
\ns\ds\les (1-\a)|\p'(Y_s)|^{\frac{2}{1-\a}}
+\frac{1+\a}{2}\bigg(e^{\frac{2\g}{1-\a}|Y_s|}+e^{\frac{2\g}{1-\a}\dbE_s[|U_{s+\d(s)}|]}\bigg).
\ea\ee
According to the definition of $\mu_1$ and $\mu_2$ defined in \rf{4.2.1.4}, we get
\bel{4.2.2.7}\ba{ll}
\ds C\dbE_t\int_t^T|\p'(Y_s)|\Big((1+|Y_s|)+(1+\dbE_s[|U_{s+\d(s)}|])\Big)ds\\
\ns\ds\les C\mu_1\dbE_t\int_t^T|\p'(Y_s)|^{\frac{2}{1-\a}}ds
+C\mu_2\dbE_t\int_t^T\bigg(e^{\frac{2\g}{1-\a}|Y_s|}+e^{\frac{2\g}{1-\a}\dbE_s[|U_{s+\d(s)}|]}\bigg)ds.
\ea\ee
Now combining \rf{4.2.2.3} and \rf{4.2.2.4}-\rf{4.2.2.7} yields
$$\ba{ll}
\ds \p(Y_t)+\frac{1}{2}\dbE_t\int_t^T|Z_s|^2ds\\
\ns\ds\les \p(\|\xi_\cd\|_{\i})+(\b+C\mu_1)\dbE_t\int_t^T|\p'(Y_s)|^{\frac{2}{1-\a}}ds\\
\ns\ds\ \ \ +C\mu_2\dbE_t\int_t^T\bigg(e^{\frac{2\g}{1-\a}|Y_s|}+e^{\frac{2\g}{1-\a}\dbE_s[|U_{s+\d(s)}|]}\bigg)ds
+\frac{1}{4L}\dbE_t\int_t^T|V_{s+\zeta(s)}|^2ds\\
\ns\ds\les \p(\|\xi_\cd\|_{\i})
+\Big[\g^{\frac{2}{\a-1}}(\b+C\mu_1)+C\mu_2\Big]\dbE_t\int_t^Te^{\frac{2\g}{1-\a}|Y_s|}ds\\
\ns\ds\ \ \ +C\mu_2\dbE_t\int_t^Te^{\frac{2\g}{1-\a}\dbE_s[|U_{s+\d(s)}|]}ds
+\frac{1}{4L}\dbE_t\int_t^T|V_{s+\zeta(s)}|^2ds.
\ea$$
Note that $U_t=\xi_t$ and $V_t=\eta_t$ when $t\in[T,T+K]$, we have
$$\ba{ll}
\ds C\mu_2\dbE_t\int_t^Te^{\frac{2\g}{1-\a}\dbE_s[|U_{s+\d(s)}|]}ds
\les C\mu_2\dbE_t\int_t^Te^{\frac{2\g}{1-\a}\esssup\limits_{\o}|U_{s+\d(s)}|}ds\\
\ns\ds\qq\qq\qq\qq\qq\qq\ \les C\mu_2L\dbE_t\int_t^{T+K}e^{\frac{2\g}{1-\a}\esssup\limits_{\o}|U_{s}|}ds\\
\ns\ds\qq\qq\qq\qq\qq\qq\ \les C\mu_2L\e e^{\frac{2\g}{1-\a}\|U_{\cd}\|_{L^\i_{\dbF}(T-\e,T)}}
+C\mu_2LK e^{\frac{2\g}{1-\a}\|\xi_{\cd}\|_{\i}}.
\ea$$
In addition, applying \rf{14.57}, one has
$$\ba{ll}
\ds \frac{1}{4L}\dbE_t\int_t^T|V_{s+\zeta(s)}|^2ds
\les \frac{1}{4}\dbE_t\int_t^{T+K}|V_{s}|^2ds
=\frac{1}{4}\dbE_t\Big(\int_t^{T}|V_{s}|^2ds+\dbE_T\int_T^{T+K}|\eta_{s}|^2ds\Big)\\
\ns\ds\qq\qq\qq\qq \q \ \les \frac{1}{4}\dbE_t\Big(\int_t^{T}|V_{s}|^2ds+\esssup_{\t\in\sT[T,T+K],\ \o\in\Omega}\dbE_\t\int_\t^{T+K}|\eta_{s}|^2ds\Big)\\
\ns\ds\qq\qq\qq\qq \q \ \les \frac{1}{4}\|V_\cd\|^2_{\cZ^2[T-\e,T]}+\frac{1}{4}\|\eta_\cd\|^2_{\cZ^2}.
\ea$$
Therefore, we have
\bel{4.2.2.8}\ba{ll}
\ds \p(Y_t)+\frac{1}{2}\dbE_t\int_t^T|Z_s|^2ds\\
\ns\ds\les \p(\|\xi_\cd\|_{\i})
+\Big[\g^{\frac{2}{\a-1}}(\b+C\mu_1)+C\mu_2\Big]\e e^{\frac{2\g}{1-\a}\|Y_{\cd}\|_{L^\i_{\dbF}(T-\e,T)}}\\
\ns\ds\ \ \ +C\mu_2L\e e^{\frac{2\g}{1-\a}\|U_{\cd}\|_{L^\i_{\dbF}(T-\e,T)}}
+C\mu_2LK e^{\frac{2\g}{1-\a}\|\xi_{\cd}\|_{\i}}
+\frac{1}{4}\|V_\cd\|^2_{\cZ^2[T-\e,T]}+\frac{1}{4}\|\eta_\cd\|^2_{\cZ^2}.
\ea\ee

\ms

$\emph{Step 2.2 Estimate of $e^{\g \|Y_\cd\|_{L^\i_{\dbF}(T-\e,T)}}$.}$

\ms

\ms

In light of the last inequality of Proposition \ref{2.1.13} we have
$$\ba{ll}
e^{\frac{3\g}{1-\a}|Y_t|}\les\dbE_t\bigg[e^{\frac{3\g}{1-\a}
e^{CT}\Big(\|\xi_\cd\|_{\i}+C\int_t^T(1+\dbE_s[|U_{s+\d(s)}|]+\dbE_s[|V_{s+\zeta(s)}|^{1+\a}])ds\Big)}\bigg].
\ea$$
Set
$$u(s)=\esssup_{\o}|U_s(\o)|.$$
Then
$$\ba{ll}
\ds e^{\frac{3\g}{1-\a}|Y_t|}\les\dbE_t\bigg[e^{\frac{3\g}{1-\a}
e^{CT}\Big(\|\xi_\cd\|_{\i}+C\int_t^T(1+u(s+\d(s))+\dbE_s[|V_{s+\zeta(s)}|^{1+\a}])ds\Big)}\bigg].\\
\ns\ds\qq\q \ \ \les\dbE_t\bigg[e^{\frac{3\g}{1-\a}
e^{CT}\Big(\|\xi_\cd\|_{\i}+C\int_t^{T+K}(1+Lu(s))ds+C\int_t^{T}\dbE_s[|V_{s+\zeta(s)}|^{1+\a}]ds\Big)}\bigg].
\ea$$
For $t\in[T-\e,T+K]$, we have
$$CL\int_t^{T+K}|u(s)|ds\les CL\e\|U_\cd\|_{L^\i_{\dbF}(T-\e,T)}+CLK\|\xi_\cd\|_{\i}.$$
Using Young's inequality, we obtain
$$
\frac{3\g }{1-\a}e^{CT}C|V_{s+\zeta(s)}|^{1+\a}
\les\frac{1}{L}\frac{1-\a}{2}\bigg(\frac{3\g CL e^{CT}}{1-\a}\Big(\frac{1+\a}{2\d}\Big)^{\frac{1+\a}{2}}\bigg)^{\frac{2}{1-\a}}
+\frac{1}{L}\d|V_{s+\zeta(s)}|^2.
$$
According to the definition of $C_{\d}$ defined in \rf{4.2.1.4}, note that \rf{4.2.1.5}, we further have
$$\ba{ll}
\ds e^{\frac{3\g}{1-\a}|Y_t|}\les C_{\d}e^{\big(\frac{3\g}{1-\a}e^{CT}(1+CLK)\|\xi_\cd\|_{\i}
+\frac{3\g}{1-\a}e^{CT}CL\e\|U_\cd\|_{L^\i_{\dbF}(T-\e,T)}\big)}\cdot
\dbE_t e^{\frac{\d}{L} \int_t^{T}\dbE_s[|V_{s+\zeta(s)}|^2]ds}.
\ea$$
Set
\bel{9.15.16.4)}M_t=\sqrt{\frac{\d}{L}}\int_0^t\dbE_s[|V_{s+\zeta(s)}|^2]^{\frac{1}{2}}dW_s.\ee
We have
$$\ba{ll}
\ds \|M\|^2_{BMO(\dbP)}=\frac{\d}{L}\sup_{\t} \dbE_\t\int_\t^T\dbE_s[|V_{s+\zeta(s)}|^2]ds\\
\ns\ds\qq\qq\q\ =\frac{\d}{L}\sup_{\t} \dbE_\t\int_\t^T|V_{s+\zeta(s)}|^2ds\\
\ns\ds\qq\qq\q\ \les\d\sup_{\t} \dbE_\t\int_\t^{T+K}|V_{s}|^2ds\\
\ns\ds\qq\qq\q\ \les\d\sup_{\t}\dbE_\t\int_\t^{T}|V_{s}|^2ds+\d\|\eta_\cd\|^2_{\cZ^2}.
\ea$$
Then we obtain that
$$\ba{ll}
\ds e^{\frac{3\g}{1-\a}|Y_t|}\les C_{\d}e^{\big(\frac{3\g}{1-\a}e^{CT}(1+CLK)\|\xi_\cd\|_{\i}
+\frac{3\g}{1-\a}e^{CT}CL\e\|U_\cd\|_{L^\i_{\dbF}(T-\e,T)}+\d\|\eta_\cd\|_{\cZ^2}^2\big)}
\cdot\dbE_t e^{\sup\limits_{\t}\dbE_\t\int_\t^{T}\d|V_{s}|^2ds}.\\
\ea$$
It follows from \rf{4.2.11} and the definition of $\mathcal{B}_\e$ that we have
$$
\sup\limits_{\t}\dbE_\t\int_\t^{T}\d|V_{s}|^2ds\les\|\sqrt{\d} V\cd W\|^2_{BMO(\dbP)}=\|\sqrt{\d} V\|^2_{\cZ^2[T-\e,T]}\les\d A<1.
$$
Then applying John-Nirenberg's inequality yields that
$$\ba{ll}
\ds e^{\frac{3\g}{1-\a}|Y_t|}\les C_{\d}\frac{e^{\big(\frac{3\g}{1-\a}e^{CT}(1+CLK)\|\xi_\cd\|_{\i}
+\d\|\eta_\cd\|_{\cZ^2}^2
+\frac{3\g}{1-\a}e^{CT}CL\e\|U_\cd\|_{L^\i_{\dbF}(T-\e,T)}\big)}}{1-\d\|V\cd W\|^2_{BMO(\dbP)}}\\
\ns\ds\qq\q \ \ \les C_{\d}\frac{e^{\big(\frac{3\g}{1-\a}e^{CT}(1+CLK)\|\xi_\cd\|_{\i}
+\d\|\eta_\cd\|^2_{\cZ^2}\big)}}{1-\d A}
e^{\big(\frac{3\g}{1-\a}e^{CT}CL\e\|U_\cd\|_{L^\i_{\dbF}(T-\e,T)}\big)}.
\ea$$
Since $3 e^{CT}CL\e\les 1$ (see the choice of $\e$ in \rf{4.2.11}) and $(U_\cd,V_\cd)\in\mathcal{B}_\e$, we have
\bel{4.9.1}\ba{ll}
\ds e^{\frac{3\g}{1-\a}\|Y_\cd\|_{L^\i_{\dbF}(T-\e,T)}}\les C_{\d}\frac{e^{\big(\frac{3\g}{1-\a}e^{CT}(1+CLK)\|\xi_\cd\|_{\i}
+\d\|\eta_\cd\|^2_{\cZ^2}\big)}}{1-\d A}
e^{\big(\frac{\g}{1-\a}\|U_\cd\|_{L^\i_{\dbF}(T-\e,T)}\big)}\\
\ns\ds\qq\qq \ \les C_{\d}\frac{e^{\big(\frac{3\g}{1-\a}e^{CT}(1+CLK)\|\xi_\cd\|_{\i}
+\d\|\eta_\cd\|^2_{\cZ^2}\big)}}{1-\d A}
\bigg(C_{\d}\frac{e^{\big(\frac{3\g}{1-\a}e^{CT}(1+CLK)\|\xi_\cd\|_{\i}
+\d\|\eta_\cd\|^2_{\cZ^2}\big)}}{1-\d A}\bigg)^{\frac{1}{2}}\\
\ns\ds\qq\qq \ \les
\bigg(C_{\d}\frac{e^{\big(\frac{3\g}{1-\a}e^{CT}(1+CLK)\|\xi_\cd\|_{\i}
+\d\|\eta_\cd\|^2_{\cZ^2}\big)}}{1-\d A}\bigg)^{\frac{3}{2}},
\ea\ee
which implies that the first half of \rf{4.2.2.1} is obtained.

\ms

\ms

$\emph{Step 2.3 Estimate of $\|Z_\cd\|^2_{\cZ^2[T-\e,T]}$.}$

\ms

\ms

From inequality \rf{4.2.2.8}, the definition of $\mu$ and $\tilde{\mu}$ in \rf{4.2.1.4} and note that $(U_\cd,V_\cd)\in\mathcal{B}_\e$, we have
$$\ba{ll}
\ds \frac{1}{2}\dbE_t\int_t^T|Z_s|^2ds\les \g^{-2}e^{{\g\|\xi_\cd\|_{\i}}}
+\mu C_{\d}\frac{e^{\big(\frac{3\g}{1-\a}e^{CT}(1+CLK)\|\xi_\cd\|_{\i}
+\d\|\eta_\cd\|^2_{\cZ^2}\big)}}{1-\d A}\e\\
\ns\ds\qq\qq\qq\qq +C\mu_2LK e^{\frac{2\g}{1-\a}\|\xi_{\cd}\|_{\i}}+\frac{1}{4}\|\eta_\cd\|^2_{\cZ^2}
+\frac{1}{4}A\\
\ns\ds\qq\qq \les\Big(\tilde{\mu}e^{\frac{2\g}{1-\a}\|\xi_{\cd}\|_{\i}}+\frac{1}{4}\|\eta_\cd\|^2_{\cZ^2}\Big)
+\mu C_{\d}\frac{e^{\big(\frac{3\g}{1-\a}e^{CT}(1+CLK)\|\xi_\cd\|_{\i}
+\d\|\eta_\cd\|^2_{\cZ^2}\big)}}{1-\d A}\e
+\frac{1}{4}A.
\ea$$
In light of \rf{4.2.1.6}, we have
$$\frac{1}{2}\|Z_\cd\|^2_{\cZ^2[T-\e,T]}\les \Big(\tilde{\mu}e^{\frac{2\g}{1-\a}\|\xi_{\cd}\|_{\i}}+\frac{1}{4}\|\eta_\cd\|^2_{\cZ^2}\Big)
+\mu C_{\d}\frac{e^{\big(\frac{3\g}{1-\a}e^{CT}(1+CLK)\|\xi_\cd\|_{\i}
+\d\|\eta_\cd\|^2_{\cZ^2}\big)}}{1-\d A}\e
+\frac{1}{4}A=\frac{1}{2}A.$$
Then
$$\|Z_\cd\|^2_{\cZ^2[T-\e,T]}\les A,$$
so we have that the other half of the desired result \rf{4.2.2.1} is obtained.

\subsubsection*{Step 3: Contraction of the quadratic solution mapping}

In this subsection, we prove that the quadratic mapping defined above is a contraction mapping.
For $(U_\cd,V_\cd)\in\mathcal{B}_\varepsilon$ and $(\widetilde{U}_\cd,\widetilde{V}_\cd)\in\mathcal{B}_\varepsilon$, set
$$(Y_\cd,Z_\cd)\triangleq\Gamma(U_\cd,V_\cd),\qq
(\widetilde{Y}_\cd,\widetilde{Z}_\cd)\triangleq\Gamma(\widetilde{U}_\cd,\widetilde{V}_\cd).$$
Then
\bel{4.2.3.1}\left\{\ba{ll}
\ds -dY_t=f(t,Y_t,Z_t,U_{t+\delta(t)},V_{t+\zeta(t)}) dt - Z_{t} dW_{t}, \qq t\in [T-\e,T]; \\
\ns\ns\ds Y_t=\xi_t, \qq  Z_{t}=\eta_t, \qq  t\in[T,T+K],
\ea\right.\ee
and
\bel{4.2.3.2}\left\{\ba{ll}
\ds -d\widetilde{Y}_t=f(t,\widetilde{Y}_t,\widetilde{Z}_t,\widetilde{U}_{t+\delta(t)},\widetilde{V}_{t+\zeta(t)}) dt
 - \widetilde{Z}_{t} dW_{t}, \qq t\in [T-\e,T]; \\
\ns\ns\ds \widetilde{Y}_t=\xi_t, \qq  \widetilde{Z}_{t}=\eta_t, \qq  t\in[T,T+K].
\ea\right.\ee
Note that in $[T,T+K]$,
\bel{9.16.16.11}Y_t=\widetilde{Y}_t=U_t=\widetilde{U}_t=\xi_t,\qq Z_t=\widetilde{Z}_t=V_t=\widetilde{V}_t=\eta_t,\qq
t\in[T,T+K].\ee
In $[T-\e,T]$, we rewrite \rf{4.2.3.1} and \rf{4.2.3.2} as the following integral form:
$$\left\{\ba{ll}
\ds Y_t=\xi_T+\int_t^Tf(s,Y_s,Z_s,U_{s+\delta(s)},V_{s+\zeta(s)}) ds -\int_t^T Z_{s} dW_{s}, \\
\ns\ds \widetilde{Y}_t=\xi_T+\int_t^Tf(s,\widetilde{Y}_s,\widetilde{Z}_s,\widetilde{U}_{s+\delta(s)},
\widetilde{V}_{s+\zeta(s)}) ds -\int_t^T \widetilde{Z}_{s} dW_{s}.
\ea\right.$$
We can define the process $\varphi_\cd$ in an obvious way such that
\bel{4.2.3.4}\left\{\ba{ll}
\ds |\f_s|\les C(1+|Z_s|+|\widetilde{Z}_s|), \\
\ns\ns\ds f(s,Y_s,Z_s,U_{s+\delta(s)},V_{s+\zeta(s)})
-f(s,Y_s,\widetilde{Z}_s,U_{s+\delta(s)},V_{s+\zeta(s)})=(Z_s-\widetilde{Z}_s)\f_s.
\ea\right.\ee
Then
$$\widetilde{W}_t\triangleq W_t-\int_0^t\f_sds$$
is a  Brownian motion corresponding to an equivalent probability measure $\widetilde{\dbP}$ defined by
$$d\widetilde{\dbP}\triangleq\mathcal{E}\Big(\f\cdot W\Big)\Big|_0^Td\dbP,$$
and from a priori estimate established in the above analysis, there exists $\widetilde{K}>0$ such that
$$\|\f\cdot W\|^2_{BMO(\dbP)}=\|\f_\cd\|^2_{\cZ^2[T-\e,T]}\les \widetilde{K}.$$
In fact, one can choose that $\widetilde{K}=3C^2T+6C^2A$.
In light of the following equation
$$\ba{ll}
\ds Y_t-\widetilde{Y}_t+\int_t^T(Z_{s}-\widetilde{Z}_{s})d\widetilde{W}_{s}\\
\ns\ds=\int_t^T\Big( f(s,Y_s,\widetilde{Z}_s,U_{s+\delta(s)},V_{s+\zeta(s)})
-f(s,\widetilde{Y}_s,\widetilde{Z}_s,\widetilde{U}_{s+\delta(s)},\widetilde{V}_{s+\zeta(s)}) \Big)ds,
\ea$$
squaring on both sides of the last equation and then taking the conditional expectation with respect to $\widetilde{\dbP}$ (denote by $\widetilde{\dbE}$), we get
\bel{4.2.3.6.2}\2n\ba{ll}
\ds |Y_t-\widetilde{Y}_t|^2+\widetilde{\dbE}_t\int_t^T|Z_{s}-\widetilde{Z}_{s}|^2ds\\
\ns\ds=\widetilde{\dbE}_t\bigg[\bigg(\int_t^T\Big( f(s,Y_s,\widetilde{Z}_s,U_{s+\delta(s)},V_{s+\zeta(s)})
-f(s,\widetilde{Y}_s,\widetilde{Z}_s,\widetilde{U}_{s+\delta(s)},\widetilde{V}_{s+\zeta(s)}) \Big)ds\bigg)^2\bigg]\\
\ns\ds\les C^2\widetilde{\dbE}_t\bigg[\bigg(\int_t^T\Big( |Y_s-\widetilde{Y}_s|
+ \dbE_s[|U_{s+\delta(s)}-\widetilde{U}_{s+\delta(s)}|]\\
\ns\ds\qq\qq\qq\q
+ \big(1+\dbE_s[|V_{s+\zeta(s)}|^{\a}+|\widetilde{V}_{s+\zeta(s)}|^{\a}]\big)
\dbE_s[|V_{s+\zeta(s)}-\widetilde{V}_{s+\zeta(s)}|] \Big)ds\bigg)^2\bigg]\\
\ns\ds\les 3C^2\widetilde{\dbE}_t\bigg[(T-t)\int_t^T|Y_s-\widetilde{Y}_s|^2ds
+(T-t)\int_t^T\dbE_s[|U_{s+\delta(s)}-\widetilde{U}_{s+\delta(s)}|^2]ds\\
\ns\ds\qq\qq \
+ \int_t^T\Big(1+\dbE_s[|V_{s+\zeta(s)}|^{\a}+|\widetilde{V}_{s+\zeta(s)}|^{\a}]\Big)^2ds
\cdot\int_t^T \dbE_s[|V_{s+\zeta(s)}-\widetilde{V}_{s+\zeta(s)}|]^2ds \bigg].
\ea\ee
On the one hand, set
\begin{equation*}
  \hat{u}(s)=\esssup_\o|U_s-\widetilde{U}_s|,
\end{equation*}
then, note \rf{9.16.16.11}, we have
\bel{9.16.20.1}\ba{ll}
\ds (T-t)\int_t^T\dbE_s[|U_{s+\delta(s)}-\widetilde{U}_{s+\delta(s)}|^2]ds
\les\e\int_t^T\hat{u}(s+\d(s))^2ds\\
\ns\ds\q\les L\e\int_t^{T+K}\hat{u}(s)^2ds
=L\e\int_t^{T}\hat{u}(s)^2ds
\les L\e^2\|U_\cd-\widetilde{U}_\cd\|_{L^{\i}_{\dbF}(T-\e,T)}^2.
\ea\ee
On the other hand, for the last term of \rf{4.2.3.6.2}, by H\"{o}lder's inequality, we have
\bel{9.16.20.2}\ba{ll}
\ds\widetilde{\dbE}_t\bigg[\int_t^T\Big(1+\dbE_s[|V_{s+\zeta(s)}|^{\a}+|\widetilde{V}_{s+\zeta(s)}|^{\a}]\Big)^2ds
\cdot\int_t^T \dbE_s[|V_{s+\zeta(s)}-\widetilde{V}_{s+\zeta(s)}|]^2ds\bigg]\\
\ns\ds\les3\widetilde{\dbE}_t \bigg[\bigg(\int_t^T\Big(1+\dbE_s[|V_{s+\zeta(s)}|^{2\a}
+|\widetilde{V}_{s+\zeta(s)}|^{2\a}]\Big)ds\bigg)^2\bigg]^{\frac{1}{2}}
\cdot\widetilde{\dbE}_t \bigg[\bigg(\int_t^T \dbE_s[|V_{s+\zeta(s)}-\widetilde{V}_{s+\zeta(s)}|]^2ds\bigg)^2\bigg]^{\frac{1}{2}}.
\ea\ee
Similar to \rf{9.15.16.4)}, set
$$M_t=\int_0^t\dbE_s[|V_{s+\zeta(s)}-\widetilde{V}_{s+\zeta(s)}|]dW_s.$$
Then, again note that \rf{9.16.16.11},  we have
$$\ba{ll}
\ds \|M\|^2_{BMO(\dbP)}=\sup_{\t} \dbE_\t\int_\t^T\dbE_s[|V_{s+\zeta(s)}-\widetilde{V}_{s+\zeta(s)}|]^2ds
 =\sup_{\t} \dbE_\t\int_\t^T|V_{s+\zeta(s)}-\widetilde{V}_{s+\zeta(s)}|^2ds\\
\ns\ds\q\les L\sup_{\t} \dbE_\t\int_\t^{T+K}|V_{s}-\widetilde{V}_{s}|^2ds
=L\sup_{\t} \dbE_\t\int_\t^{T}|V_{s}-\widetilde{V}_{s}|^2ds
 \les L\|V_\cd-\widetilde{V}_\cd\|^2_{\cZ^2[T-\e,T]}.
\ea$$
Hence from Proposition \ref{2.1.11.0} and Proposition \ref{2.1.11}, there exist $L_4>0$ and $c_2>0$ such that
\bel{9.16.21.0}\ba{ll}
\ds \widetilde{\dbE}_t \bigg[\bigg(\int_t^T \dbE_s[|V_{s+\zeta(s)}-\widetilde{V}_{s+\zeta(s)}|]^2ds\bigg)^2\bigg]^{\frac{1}{2}}
\les\|M\|^2_{BMO_4(\widetilde{\dbP})}\\
\ns\ds\q\les L^2_4\|M\|^2_{BMO(\widetilde{\dbP})}
\les L^2_4 c^2_2\|M\|^2_{BMO(\dbP)}
\les LL^2_4 c^2_2\|V_\cd-\widetilde{V}_\cd\|^2_{\cZ^2[T-\e,T]}.
\ea\ee
Using H\"{o}lder's inequality, for $t\in[T-\varepsilon,T]$,
$$\ba{ll}
\ds \widetilde{\dbE}_t\bigg[\bigg(
\int_t^T\big(1+\dbE_s[|V_{s+\zeta(s)}|^{2\a}]+\dbE_s[|\widetilde{V}_{s+\zeta(s)}|^{2\a}]\big)ds
\bigg)^2\bigg]^{\frac{1}{2}}\\
\ns\ds \les  \widetilde{\dbE}_t\bigg[\bigg(
\e+\e^{1-\a}\Big(\int_t^T\dbE_s[|V_{s+\zeta(s)}|^{2}]ds\Big)^{\a}
+\e^{1-\a}\Big(\int_t^T\dbE_s[|\widetilde{V}_{s+\zeta(s)}|^{2}]ds\Big)^{\a}
\bigg)^2\bigg]^{\frac{1}{2}}\\
\ns\ds \les  \e^{1-\a}\widetilde{\dbE}_t\bigg[\bigg(
\e^{\a}+\Big(\int_t^T\dbE_s[|V_{s+\zeta(s)}|^{2}]ds\Big)^{\a}
+\Big(\int_t^T\dbE_s[|\widetilde{V}_{s+\zeta(s)}|^{2}]ds\Big)^{\a}
\bigg)^2\bigg]^{\frac{1}{2}}\\
\ns\ds \les  \e^{1-\a}\widetilde{\dbE}_t\bigg[\bigg(
T^{\a}+2-2\a+\a\int_t^T\dbE_s[|V_{s+\zeta(s)}|^{2}]ds
+\a\int_t^T\dbE_s[|\widetilde{V}_{s+\zeta(s)}|^{2}]ds
\bigg)^2\bigg]^{\frac{1}{2}}\\
\ns\ds \les  \e^{1-\a}\bigg(
T^{\a}+2-2\a
+\a\widetilde{\dbE}_t\[\(\int_t^T\dbE_s[|V_{s+\zeta(s)}|^{2}]ds\)^2\]^{\frac{1}{2}}
+\a\widetilde{\dbE}_t\[\(\int_t^T\dbE_s[|\widetilde{V}_{s+\zeta(s)}|^{2}]ds
\)^2\]^{\frac{1}{2}}\bigg).
\ea$$
Similar to the above discussion, and note that \rf{9.16.16.11}, we have
$$\ba{ll}
\ds \widetilde{\dbE}_t \bigg[\bigg(\int_t^T \dbE_s[|V_{s+\zeta(s)}|]^2ds\bigg)^2\bigg]^{\frac{1}{2}}
\les LL^2_4 c^2_2\big(\|V_\cd\|^2_{\cZ^2[T-\e,T]}+\|\eta\|_{\cZ^2}^2\)
\les LL^2_4 c^2_2\big(A+\|\eta\|_{\cZ^2}^2\),\\
\ns\ds  \widetilde{\dbE}_t \bigg[\bigg(\int_t^T \dbE_s[|\widetilde{V}_{s+\zeta(s)}|]^2ds\bigg)^2\bigg]^{\frac{1}{2}}
\les LL^2_4 c^2_2\big(\|\widetilde{V}_\cd\|^2_{\cZ^2[T-\e,T]}+\|\eta\|_{\cZ^2}^2\)
\les LL^2_4 c^2_2\big(A+\|\eta\|_{\cZ^2}^2\).
\ea$$
Hence
\bel{9.16.21.1}\ba{ll}
\ds \widetilde{\dbE}_t\bigg[\bigg(
\int_t^T\big(1+\dbE_s[|V_{s+\zeta(s)}|^{2\a}]+\dbE_s[|\widetilde{V}_{s+\zeta(s)}|^{2\a}]\big)ds
\bigg)^2\bigg]^{\frac{1}{2}}\\
\ns\ns\ds \les  \e^{1-\a}\Big(T^{\a}+2+2\a L L^4_4c^2_2\big(A+\|\eta_\cdot\|^2_{\cZ^2}\big)\Big).
\ea\ee
Combining \rf{9.16.20.2}, \rf{9.16.21.0} and \rf{9.16.21.1}, one has
\bel{9.16.20.3}\ba{ll}
\ds\widetilde{\dbE}_t\bigg[\int_t^T\Big(1+\dbE_s[|V_{s+\zeta(s)}|^{\a}+|\widetilde{V}_{s+\zeta(s)}|^{\a}]\Big)^2ds
\cdot\int_t^T \dbE_s[|V_{s+\zeta(s)}-\widetilde{V}_{s+\zeta(s)}|]^2ds\bigg]\\
\ns\ds\les 3\e^{1-\a}\Big(T^{\a}+2+2\a L L^4_4c^2_2\big(A+\|\eta_\cdot\|^2_{\cZ^2}\big)\Big)\cdot
LL^2_4 c^2_2\|V_\cd-\widetilde{V}_\cd\|^2_{\cZ^2[T-\e,T]}.
\ea\ee
Again, combining \rf{4.2.3.6.2}, \rf{9.16.20.1} and \rf{9.16.20.3}, we obtain that
\bel{4.2.3.6}\2n\ba{ll}
\ds |Y_t-\widetilde{Y}_t|^2+\widetilde{\dbE}_t\int_t^T|Z_{s}-\widetilde{Z}_{s}|^2ds\\
\ns\ds\les 3C^2\e^2\|Y_\cd-\widetilde{Y}_\cd\|_{L^{\i}_{\dbF}(T-\e,T)}^2
+3C^2L\e^2\|U_\cd-\widetilde{U}_\cd\|_{L^{\i}_{\dbF}(T-\e,T)}^2\\
\ns\ds\ \ \
+ 9C^2LL^2_4 c^2_2\Big(T^{\a}+2+2\a L L^4_4c^2_2\big(A+\|\eta_\cdot\|^2_{\cZ^2}\big)\Big)\e^{1-\a}
\cdot \|V_\cd-\widetilde{V}_\cd\|^2_{\cZ^2[T-\e,T]}.
\ea\ee
In view of estimates \rf{2.1.12}, we have for $t\in[T-\e,T]$,
$$\ba{ll}
\ds (1-3C^2\e^2)\|Y_\cd-\widetilde{Y}_\cd\|^2_{L^{\i}_{\dbF}(T-\e,T)}+c_1^2\|Z_\cd-\widetilde{Z}_\cd\|^2_{\cZ^2[T-\e,T]}\\
\ns\ds\les
3C^2\e^2\|Y_\cd-\widetilde{Y}_\cd\|_{L^{\i}_{\dbF}(T-\e,T)}^2
+ \widetilde{C}\e^{1-\a}
\cdot \|V_\cd-\widetilde{V}_\cd\|^2_{\cZ^2[T-\e,T]},
\ea$$
where
$$\widetilde{C}=9C^2LL^2_4 c^2_2\Big(T^{\a}+2+2\a L L^4_4c^2_2\big(A+\|\eta_\cdot\|^2_{\cZ^2}\big)\Big).$$
It is then standard to show that there exists a small positive constant $\e$ such that the mapping $\G$ is a contraction on the previously given set $\mathcal{B}_{\e}$. The desired result is obtained.
%
\endpf

\subsection{Global Solution}

Based on the previous result concerning the local solution of BSDE \rf{1.1}, we investigate the global adapted solution of a class of BSDE \rf{1.1}. Before presenting the main result, we present the assumption:
\bas{4.1.4} \rm
Assume that  $f(s,\o,z):[0,T]\ts\Om\ts\dbR\ra\dbR$ is an $\mathcal{F}_s$-adapted scalar-valued generator, and
for all $s\in[0,T]$, $h(s,\o,y,\xi,\eta):\Om\ts\dbR\ts L^2_{\cF_r}(\Om;\dbR)\ts
L^2_{\cF_{\bar{r}}}(\Om;\dbR)\ra L^2_{\cF_s}(\Om;\dbR)$, where $r,\bar{r}\in[s,T+K]$.
Let $C$ be a positive constant. For all $s\in[0,T]$, $y,\bar{y},z,\bar{z}\in\dbR$,
$\xi_{\cd},\bar{\xi}_{\cd},\eta_{\cd},\bar{\eta}_{\cd}\in L^2_{\dbF}(s,T+K;\dbR)$, we have
$$\ba{ll}
\ds |f(s,z)|\les C(1+|z|^2),\qq |h(s,y,\xi_r,\eta_{\bar{r}})|\les C\big(1+|y|+\dbE_s[\xi_r]\big);\\
\ns\ds |f(s,z)-f(s,\bar{z})|\les C(1+|z|+|\bar{z}|)|z-\bar{z}|;\\
\ns\ds |h(s,y,\xi_r,\eta_{\bar{r}})-h(s,\bar{y},\bar{\xi}_r,\bar{\eta}_{\bar{r}})|
\les C\big(|y-\bar{y}|+\dbE_s\big[|\xi_r-\bar{\xi}_r|+|\eta_{\bar{r}}-\bar{\eta}_{\bar{r}}|\big]\big).
\ea$$
\eas

\begin{example} \rm
Assumption \ref{4.1.4} requires that $h(\cd)$ should be bounded with respect to $\eta_\cd$. For example, the generators
$$\ba{ll}
\ds f(s,z)=1+|z|^2,\qq h(s,y,\xi_r,\eta_{\bar r})=1+|y|+\dbE_s\big[|\xi_r|+|\sin(\eta_{\bar r})|\big],\\
\ns\ds\qq\qq  \forall s\in[0,T], \q y,z\in\dbR, \q \xi_{\cd},\eta_{\cd}\in L^2_{\dbF}(s,T+K;\dbR), \q r,\bar{r}\in[s,T+K]
\ea$$
satisfy such an assumption.
\end{example}

In this subsection, the main result is the following theorem, which concerns the solvability of global adapted solution of BSDE \rf{1.1}.

\bt{4.1.5}\sl Let Assumptions \ref{4.1.2} and \ref{4.1.4} hold. Then, the following BSDE
\bel{4.1.6}\left\{\ba{ll}
\ds -dY_t=[f(t,Z_t)+h(t,Y_t,Y_{t+\delta(t)},Z_{t+\zeta(t)})]dt-Z_{t}dW_{t}, \qq t\in [0,T]; \\
\ns\ds Y_t=\xi_t, \qq  Z_{t}=\eta_t, \qq  t\in[T,T+K],
\ea\right.\ee
has a unique adapted solution $(Y_\cd,Z_\cd)$ on $[0,T+K]$ such that $Y_\cd$ is bounded and $Z\cd W$ is a $BMO(\dbP)$-martingale.
\et

\it Proof. \rm
Since when $t\in[T,T+K]$, the values of $(Y_\cd,Z_\cd)$ of BSDE \rf{4.1.6} are determined, we essentially need to prove that \rf{4.1.6} has a unique adapted solution on $[0,T]$. In order to do this, we first consider the equation \rf{4.1.6} on $[T-\theta_{\l},T+K]$  for a positive constant $\theta_{\l}$ that will be determined later.

\ms

For simplicity of presentation, we rewrite BSDE \rf{4.1.6} as the following integral form,
\bel{4.3.0}\left\{\ba{ll}
\ds Y_t=\xi_T+\int_t^T\Big(f(s,Z_s)+h(s,Y_s,Y_{s+\d(s)},Z_{s+\zeta(s)})\Big)ds-\int_t^TZ_{s}dW_{s},\q~t\in[0,T];\\
\ns\ds Y_t=\xi_t, \q~  Z_{t}=\eta_t, \q~  t\in[T,T+K].
\ea\right.\ee
Note that there exists a constant $\widetilde{C}>0$ such that $\|\xi_\cd\|^2_{\i}\les\widetilde{C}$. It follows from Assumption \ref{4.1.4} that we have the following inequality
\bel{19.3.23.1}\ba{ll}
\ds |2xh(s,y,\xi_r,\eta_{\bar r})|
\les\widetilde{C}+\widetilde{C}|x|^2+\widetilde{C}|y|^2+\widetilde{C}\dbE_s[|\xi_r|^2], \\
\ns\ns\ds\qq\forall s\in[0,T], x,y\in\dbR,\xi_{\cd},\eta_{\cd}\in L^2_{\dbF}(s,T+K;\dbR),r,\bar{r}\in[s,T+K].
\ea\ee
Let $\a(\cd)$ be the unique solution of the following equation,
$$\a(t)=\widetilde{C}+\int_t^{T+K}\widetilde{C}ds+\int_t^{T+K}\Big((1+L)\widetilde{C}+\widetilde{C}\Big)\a(s)ds,\q~ t\in[0,T+K].$$
It is easy to observe that $\a(\cdot)$ is a continuously decreasing function in $t$ and
$$\a(t)=\widetilde{C}+\int_t^{T+K}\widetilde{C}\Big(1+(1+L)\a(s)\Big)ds+\widetilde{C}\int_t^{T+K}\a(s)ds,\q~ t\in[0,T+K].$$
Define
$$\l\triangleq\sup_{t\in[0,T+K]}\a(t)=\a(0).$$
As $\|\xi_\cd\|^2_{\i}\les\widetilde{C}\les\l$, Theorem \ref{4.1.3} shows that there exists $\theta_{\l}>0$, which only depends on $\l$, such that BSDE \rf{4.3.0} has a local solution on $[T-\theta_{\l},T+K]$ and it can be constructed through the Picard iteration. Without loss of generality, we suppose that $f(s,0)=0,\ s\in[0,T]$.

\ms

Let us consider the Picard iteration:
$$\left\{\ba{ll}
\ds Y^0_t=\xi_T+\int_t^TZ^0_sdW_s,\q~t\in[T-\theta_{\l},T];\\
\ns\ds Y^0_t=\xi_t, \q~  Z^0_{t}=\eta_t, \q~  t\in[T,T+K],
\ea\right.$$
and for $i\ges0$,
$$\left\{\ba{ll}
\ds Y^{i+1}_t=\xi_T+\int_t^T\big[f(s,Z^{i+1}_s)+h(s,Y^i_s,Y^i_{s+\delta(s)},Z^i_{s+\zeta(s)})\big]ds
-\int_t^TZ^{i+1}_sdW_s\\
\ns\ds\q\q\ =\xi_T+\int_t^Th(s,Y^i_s,Y^i_{s+\delta(s)},Z^i_{s+\zeta(s)})ds
-\int_t^TZ^{i+1}_sd\widetilde{W}_s^{i+1},\q~ t\in[T-\theta_{\l},T];\\
\ns\ns\ds Y^{i+1}_t=\xi_t, \q~ Z^{i+1}_{t}=\eta_t, \q~ t\in[T,T+K],
\ea\right.$$
where
$$f(s,Z^{i+1}_s)=Z^{i+1}_s\f^{i+1}_s,$$
and
$$\widetilde{W}_t^{i+1}=W_t-\int_0^t\f^{i+1}_s\mathbf{1}_{[T-\theta_{\l},T]}(s)ds$$
is a Brownian motion corresponding to an equivalent probability measure $\dbP^{i+1}$ which is denoted by $\widetilde{\dbP}$
hereafter for simplicity, and whose expectation is denoted by $\widetilde{\dbE}$.

\ms

By induction, we show that the following inequality holds for $i\ges 0$,
\bel{4.3.2}|Y^i_t|^2\les \a(t),\q~ t\in[T-\th_{\l},T+K].\ee
Actually, it is easy to see that $|Y^0_t|^2\les \a(t)$, and suppose that $|Y^i_t|^2\les \a(t)$ for
$t\in[T-\th_{\l},T+K]$. Then we just need to prove that $|Y^{i+1}_t|^2\les \a(t)$ for
$t\in[T-\th_{\l},T+K]$. Using It\^{o}'s formula and the inequality \rf{19.3.23.1}, it is straightforward to deduce the following estimate for $r\in[T-\theta_{\l},t]$,
\bel{4.3.1}\ba{ll}
\ds \widetilde{\dbE}_r[|Y^{i+1}_t|^2]+\widetilde{\dbE}_r\int_t^T|Z^{i+1}_s|^2ds\\
\ns\ds = \widetilde{\dbE}_r[|\xi_T|^2]
+\widetilde{C}\widetilde{\dbE}_r\int_t^T2Y^{i+1}_sh(s,Y^i_s,Y^i_{s+\delta(s)},Z^i_{s+\zeta(s)})ds\\
\ns\ds\les \widetilde{C}+\int_t^T\widetilde{C}ds
+\widetilde{C}\widetilde{\dbE}_r\int_t^T|Y^{i+1}_s|^2ds
+\widetilde{C}\widetilde{\dbE}_r\int_t^T|Y^{i}_s|^2ds
+\widetilde{C}\widetilde{\dbE}_r\int_t^T\dbE_s[|Y^{i}_{s+\delta(s)}|^2]ds\\
\ns\ds\les \widetilde{C}+\int_t^T\widetilde{C}ds
+\widetilde{C}\widetilde{\dbE}_r\int_t^T|Y^{i+1}_s|^2ds
+\widetilde{C}\int_t^T\a(s)ds
+\widetilde{C}\int_t^T\a(s+\d(s))ds\\
\ns\ds\les \widetilde{C}+\widetilde{C}\int_t^T\widetilde{\dbE}_r[|Y^{i+1}_s|^2]ds
+\widetilde{C}\int_t^{T}\Big(1+\a(s)\Big)ds
+\widetilde{C}L\int_t^{T+K}\a(s)ds\\
\ns\ds\les \widetilde{C}+\widetilde{C}\int_t^{T+K}\widetilde{\dbE}_r[|Y^{i+1}_s|^2]ds
+\widetilde{C}\int_t^{T+K}\Big(1+(1+L)\a(s)\Big)ds.
\ea\ee
Then we get
$$\widetilde{\dbE}_r[|Y^{i+1}_t|^2]\les \widetilde{C}
+\widetilde{C}\int_t^{T+K}\Big(1+(1+L)\a(s)\Big)ds+\widetilde{C}\int_t^{T+K}\widetilde{\dbE}_r[|Y^{i+1}_s|^2]ds.$$
It follows from the comparison theorem that we have
$$\widetilde{\dbE}_r[|Y^{i+1}_t|^2]\les \a(t),\q~ t\in[T-\th_{\l},T+K].$$
Setting $r=t$, we obtain
$$|Y^{i+1}_t|^2\les \a(t),\q~ t\in[T-\th_{\l},T+K].$$
Therefore, inequality \rf{4.3.2} holds.

\ms

Since $Y_t=\lim\limits_iY^i_t$, the constructed local solution $(Y_\cd,Z_\cd)$ in $[T-\th_{\l},T+K]$ satisfies the following estimate,
$$|Y_t|^2\les \a(t),\q~ t\in[T-\th_{\l},T+K].$$
In particular, $|Y_{T-\th_{\l}}|^2\les\a(T-\th_{\l})\les\l$.

\ms

Taking $T-\th_{\l}$ as the terminal time and $Y_{T-\th_{\l}}$ as the terminal value, it follows from Theorem \ref{4.1.3} that BSDE \rf{4.3.0} has a local solution $(Y,Z)$ on $[T-2\th_{\l},T-\th_{\l}+K]$ through the Picard iteration. Also, using the Picard iteration and the fact that $|Y_{T-\th_{\l}}|^2\les\a(T-\th_{\l})$, we deduce that $|Y_t|^2\les \a(t)$ for $t\in[T-2\th_{\l},T-\th_{\l}+K]$. Repeating the preceding process, we can extend the pair $(Y_\cd,Z_\cd)$ to the whole interval $[0,T+K]$ within a finite of steps such that $Y_\cd$ is uniformly bounded by $\l$. We next show that $Z\cdot W$ is a
$BMO(\dbP)$-martingale.

\ms

Similar to the proof of the inequality \rf{4.2.2.2}, we have (with $\g=2C$)
$$\ba{ll}
\ds \p(Y_t)+\frac{1}{2}\dbE_t\int_t^T|Z_s|^2ds\\
\ns\ds\les \p(\|\xi_\cd\|_{\i})+C\dbE_t\int_t^T|\p'(Y_s)|\Big(2+|Y_s|+\dbE_s[Y_{s+\d(s)}]\Big)ds\\
\ns\ds\les \p(\|\xi_\cd\|_{\i})+C\dbE_t\int_t^T\p'(\l)\Big(2+|Y_s|+\dbE_s[Y_{s+\d(s)}]\Big)ds\\
\ns\ds\les \p(\|\xi_\cd\|_{\i})+C\p'(\l)\dbE_t\int_t^{T+K}\Big(2+(1+L)|Y_s|\Big)ds\\
\ns\ds\les \p(\|\xi_\cd\|_{\i})+C\p'(\l)\int_t^{T+K}\Big(2+(1+L)\l\Big)ds.\\
\ea$$
Hence, we have
$$\|Z\cdot W\|^2_{BMO(\dbP)}\les 2\p(\|\xi_\cd\|_{\i})+2C\p'(\l)\big(2+(1+L)\l\big)(T+K),$$
which implies that  $Z\cdot W$ is a $BMO(\dbP)$-martingale.

\ms

Finally, we shall prove the uniqueness. Similarly, we first consider the uniqueness on the interval $[T-\e,T]$.
Let $(Y_\cd,Z_\cd)$  and $(\widetilde{Y}_\cd,\widetilde{Z}_\cd)$ be two adapted solutions. Then we have (recall that $\f_\cd$ is defined by \rf{4.2.3.4})
$$\ba{ll}
\ds Y_t-\widetilde{Y}_t=\int_t^T\Big([f(s,Z_s)-f(s,\widetilde{Z}_s)]+[h(s,Y_s,Y_{s+\delta(s)},Z_{s+\zeta(s)})
-h(s,\widetilde{Y}_s,\widetilde{Y}_{s+\delta(s)},\widetilde{Z}_{s+\zeta(s)})]\Big) ds\\
\ns\ds\qq\qq \ -\int_t^T (Z_{s}-\widetilde{Z}_{s}) dW_{s}\\
\ns\ds\qq\q \ =\int_t^T\Big(h(s,Y_s,Y_{s+\delta(s)},Z_{s+\zeta(s)})
-h(s,\widetilde{Y}_s,\widetilde{Y}_{s+\delta(s)},\widetilde{Z}_{s+\zeta(s)})\Big) ds
-\int_t^T (Z_{s}-\widetilde{Z}_{s}) d\widetilde{W}_{s}.
\ea$$
Similar to the two inequalities in \rf{4.2.3.6}, for any stopping time $\t$ taking values in $[T-\e,T]$, we have
$$\ba{ll}
\ds |Y_\t-\widetilde{Y}_\t|^2+\widetilde{\dbE}_\t\int_\t^T|Z_{s}-\widetilde{Z}_{s}|^2ds\\
\ns\ds=\widetilde{\dbE}_\t\bigg[\bigg(\int_\t^T\Big( h(s,Y_s,Y_{s+\delta(s)},Z_{s+\zeta(s)})
-h(s,\widetilde{Y}_s,\widetilde{Y}_{s+\delta(s)},\widetilde{Z}_{s+\zeta(s)}) \Big)ds\bigg)^2\bigg]\\
\ns\ds\les C^2\widetilde{\dbE}_\t\bigg[\bigg(\int_\t^T\Big( |Y_s-\widetilde{Y}_s|
+ \dbE_s[|Y_{s+\delta(s)}-\widetilde{Y}_{s+\delta(s)}|]
+ \dbE_s[|Z_{s+\zeta(s)}-\widetilde{Z}_{s+\zeta(s)}|] \Big)ds\bigg)^2\bigg]\\
\ns\ds\les 3C^2\e\widetilde{\dbE}_\t\bigg[\int_\t^T|Y_s-\widetilde{Y}_s|^2ds
+\int_\t^T\dbE_s[|Y_{s+\delta(s)}-\widetilde{Y}_{s+\delta(s)}|^2]ds
+\int_\t^T \dbE_s[|Z_{s+\zeta(s)}-\widetilde{Z}_{s+\zeta(s)}|]^2ds \bigg].
\ea$$
On the one hand,
$$\ba{ll}
\ds \int_\t^T\dbE_s[|Y_{s+\delta(s)}-\widetilde{Y}_{s+\delta(s)}|^2]ds
\les\int_\t^T\esssup_{\o}|Y_{s+\delta(s)}-\widetilde{Y}_{s+\delta(s)}|^2ds \les L\int_\t^{T+K}\esssup_{\o}|Y_{s}-\widetilde{Y}_{s}|^2ds\\
\ns\ds\qq\qq\qq\qq\q = L\int_\t^{T}\esssup_{\o}|Y_{s}-\widetilde{Y}_{s}|^2ds
\les L\e\|Y_\cd-\widetilde{Y}_\cd\|_{L^\i_{\dbF}(T-\e,T)}^2.
\ea$$
On the other hand,
$$\ba{ll}
\ds \|Z-\widetilde{Z}\|^2_{BMO(\dbP)}
=\sup_{\t}\dbE_\t\int_\t^T \dbE_s[|Z_{s+\zeta(s)}-\widetilde{Z}_{s+\zeta(s)}|]^2ds
=\sup_{\t}\dbE_\t\int_\t^T|Z_{s+\zeta(s)}-\widetilde{Z}_{s+\zeta(s)}|^2ds\\
\ns\ds\q \les L\sup_{\t}\dbE_\t\int_\t^{T+K}|Z_{s}-\widetilde{Z}_{s}|^2ds
= L\sup_{\t}\dbE_\t\int_\t^{T}|Z_{s}-\widetilde{Z}_{s}|^2ds
\les L\|Z_\cd-\widetilde{Z}_\cdot\|^2_{\cZ^2[T-\e,T]}.
\ea$$
Using H\"{o}lder's inequality, and similar to \rf{9.16.21.0}, there exist $L_4>0$ and $c_2>0$ such that
$$\ba{ll}
\ds\widetilde{\dbE}_\t\bigg[\int_\t^T \dbE_s[|Z_{s+\zeta(s)}-\widetilde{Z}_{s+\zeta(s)}|]^2ds \bigg]
\les\e\widetilde{\dbE}_\t\bigg[\(\int_\t^T \dbE_s[|Z_{s+\zeta(s)}-\widetilde{Z}_{s+\zeta(s)}|]^2ds\)^2 \bigg]^{\frac{1}{2}}\\
\ns\ds \les\e\|Z-\widetilde{Z}\|^4_{BMO(\widetilde{\dbP})}\les\e L^2_4\|Z-\widetilde{Z}\|^2_{BMO(\widetilde{\dbP})}
\les\e L^2_4c_2^2\|Z-\widetilde{Z}\|^2_{BMO(\dbP)}
\les\e LL^2_4c_2^2\|Z_\cd-\widetilde{Z}_\cdot\|^2_{\cZ^2[T-\e,T]}.
\ea$$
So we obtain that
$$\ba{ll}
\ds |Y_\t-\widetilde{Y}_\t|^2+\widetilde{\dbE}_\t\int_\t^T|Z_{s}-\widetilde{Z}_{s}|^2ds\\
\ns\ds\les3C^2(1+L)\e^2\|Y_\cd-\widetilde{Y}_\cd\|_{L^\i_{\dbF}(T-\e,T)}^2
+3C^2LL^2_4c_2^2\e^2\|Z_\cd-\widetilde{Z}_\cdot\|^2_{\cZ^2[T-\e,T]}.
\ea$$
Therefore, again, note that $Y_t-\widetilde{Y}_t\equiv 0$ and $Z_t-\widetilde{Z}_t\equiv 0$ in $[T,T+K]$,
we have (on the interval $[T-\e,T+K]$),
$$\ba{ll}
\ds \|Y_\cd-\widetilde{Y}_\cd\|_{L^\i_{\dbF}(T-\e,T)}^2+c_1^2\|Z_\cd-\widetilde{Z}_\cd\|^2_{\cZ^2[T-\e,T]}\\
\ns\ns\ds\les 3C^2(1+L)\e^2\|Y_\cd-\widetilde{Y}_\cd\|_{L^\i_{\dbF}(T-\e,T)}^2
+3C^2LL^2_4c_2^2\e^2\|Z_\cd-\widetilde{Z}_\cd\|^2_{\cZ^2[T-\e,T]}.
\ea$$
Since
$$|\f_\cd|\les C(1+|Z_\cd|+|\widetilde{Z}_\cd|),$$
the two generic constants $c_1$ and $c_2$ only depend on the sum
$$\|Z_\cd\|^2_{\cZ^2[0,T+K]}+\|\widetilde{Z}_\cd\|^2_{\cZ^2[0,T+K]}.$$
Then when $\e$ is sufficiently small, we conclude that $Y_\cd=\widetilde{Y}_\cd$ and  $Z_\cd=\widetilde{Z}_\cd$ on $[T-\e,T]$. Repeating iteratively within a finite of steps, we can derive the uniqueness on the interval $[0,T]$.
\endpf
%
%
%
%
%
%
%
%
%
%
%
%
%
%
%

\section{Scalar Case: Unbounded Terminal Value}

In this section, we study a class of one-dimensional anticipated BSDE with quadratic growth and with unbounded terminal value, i.e., the terminal value $\xi_\cd$ belongs to $S_\dbF^2(0,T+K;\dbR)$.
Before going further, we would like to show a lemma first. In the following, for a function $g(\cd)$, we use  $g_\mu(\cd)$ to denote the partial derivative of $g(\cd)$ at $\mu$  with $\mu=t,y$. And $g_{\mu\mu}(\cd)$ is defined similarly.
\bl{8.8.20.34} \sl
Let $\l(t,y):[0,T]\ts\dbR\ra\dbR$ be a continuous function satisfying the following items:
\begin{enumerate}[~~\,\rm(i)]
\item $\sup\limits_{t\in[0,T]}|\l(t,y)|\in L^1(\dbR)$;
\item $\l_{t}(\cd)$ is continuous such that $\sup\limits_{t\in[0,T]}|\l_t(t,y)|\in L^1(\dbR)$.
\end{enumerate}
Then the function
$$\p(t,y)\triangleq\int_0^y\exp\Big(2\int_0^s\l(t,r)dr\Big)ds,\qq(t,y)\in[0,T]\ts\dbR,$$
satisfies the differential equation $\p_{yy}(t,y)-2\l(t,y)\p_{y}(t,y)=0$,
and has the following properties:
\begin{enumerate}[~~\,\rm(i)]
\item [1.] $\p(t,\cd)$ is a one to one function from $\dbR$ onto $\dbR$,
and both $\p(t,\cd)$ and its inverse $\p^{-1}(t,\cd)$ are smooth functions for all $t\in[0,T]$. Moreover, $\p(t,\cd)$
and $\p^{-1}(t,\cd)$ are globally Lipschitz in $y$ for every $t\in[0,T]$,
i.e., there exists constant $L$ such that
$$|\p(t,y)-\p(t,\bar{y})|\les L|y-\bar{y}|,\q~
|\p^{-1}(t,y)-\p^{-1}(t,\bar{y})|\les L|y-\bar{y}|,\q~\forall~t\in[0,T],y,\bar{y}\in\dbR.$$
\item [2.] $\p_{y}(\cd)$ is a positive bounded function,
and $\p_{t}(\cd)$ is continuous in $t$ and Lipschitz continuous in $y$.
\end{enumerate}
\el

\it Sketch of proof. \rm
From the definition, one has
$$\ba{ll}
\ds\p_t(t,y)=\int_0^y\exp\Big(2\int_0^s\l(t,r)dr\Big)\Big(2\int_0^s\l_t(t,r)dr\Big)ds\\
\ns\ds\p_{y}(t,y)=\exp\Big(2\int_0^y\l(t,r)dr\Big),\q~
\p_{yy}(t,y)=2\l(t,y)\exp\Big(2\int_0^y\l(t,r)dr\Big).\ea$$
Since $\l(\cd)$ is continuous, clearly, $\p(\cd)$ and $\p^{-1}(\cd)$ are continuous,
one to one, strictly increasing function, and belongs
to $\cC^2(\dbR)$, moreover, $\p_{yy}(t,y)-2h(t,y)\p_{y}(t,y)=0$.

\ms

Since $\l(t,\cd)$  is integrable  for every $t\in[0,T]$,
it is directly to see that $\p_{y}(\cd)$ is bounded and positive,
and therefore $\p(t,\cd)$ and $\p^{-1}(t,\cd)$ are globally Lipschitz with respect to $y$.
Moreover, one can check that $\p_{t}(\cd)$ is continuous in $t$
and Lipschitz continuous in $y$.
\endpf

\ms

Now, consider the following equation,
\bel{8.8.20.25}\left\{\ba{ll}
\ds -dY_t=\Big(f(t,Y_t,Y_{t+\d(t)},Z_{t+\z(t)})+\l(t,Y_t)Z_t^2\Big)dt-Z_{t}dW_{t},\q~ t\in [0,T]; \\
\ns\ns\ds Y_t=\xi_t, \q~ Z_{t}=\eta_t, \q~ t\in[T,T+K],
\ea\right.\ee
where the deterministic continuous functions $\delta(\cdot)$ and $\zeta(\cdot)$ satisfy  \rf{14.55} and \rf{14.56}.
In order to establish the existence and uniqueness of the above equation, we need the following assumptions.
\bas{8.8.20.49} \rm
Assume that for all $s\in[0,T]$, $f(s,\o,y,\xi,\eta):\Om\ts\dbR\ts\ts L^2_{\cF_r}(\Om;\dbR)\ts
L^2_{\cF_{\bar{r}}}(\Om;\dbR)\ra L^2_{\cF_s}(\Om;\dbR)$, where $r,\bar{r}\in[s,T+K]$.
Let $C$ be a positive constant such that for all $s\in[0,T]$, $y,\bar{y}\in\dbR$,
$\xi_{\cd},\bar{\xi}_{\cd},\eta_{\cd},\bar{\eta}_{\cd}\in L^2_{\dbF}(s,T+K;\dbR)$, $r,\bar{r}\in[s,T+K]$, we have
$$\ba{ll}
\ds\qq\qq\qq \ \ |f(s,y,\xi_{r},\eta_{\bar{r}})|\les C,\\
\ns\ds |f(s,y,\xi_{r},\eta_{\bar{r}})-f(s,\bar{y},\bar{\xi}_{r},\bar{\eta}_{\bar{r}})|
\les C\big(|y-\bar{y}|+\dbE_s\big[|\xi_{r}-\bar\xi_{r}|+|\eta_{\bar{r}}-\bar{\eta}_{\bar{r}}|\big]\big).\\
\ea$$
\eas

\bas{5.2.3} \rm
The terminal value
$\xi_{\cd} \in S^2_\dbF(T,T+K;\dbR)$ and $\eta_{\cd}\in L^2_\dbF(T,T+K;\dbR)$.
\eas

\begin{theorem}\label{8.8.21.2} \sl
Let  Assumptions \ref{8.8.20.49} and \ref{5.2.3} hold, and suppose that $\l(\cd)$ satisfy the condition of Lemma \ref{8.8.20.34}. Then  BSDE \rf{8.8.20.25} admits a unique adapted solution $(Y_\cd,Z_\cd)$ in
$S_\dbF^2(0,T+K;\dbR)\times L^2_\dbF(0,T+K;\dbR)$.
\end{theorem}

\begin{example} \rm
The following coefficients satisfy the assumptions of Theorem \ref{8.8.21.2}: For every $s\in[0,T]$, $y\in\dbR$,
$\xi_{\cd},\eta_{\cd}\in L^2_{\dbF}(s,T+K;\dbR)$, $r,\bar{r}\in[s,T+K]$,
$$\ba{ll}
\ds f(s,y,\xi_r,\eta_{\bar r})=1+|\sin(y)|+\dbE_s\big[|\cos(\xi_r)|+|\cos(\eta_{\bar r})|\big],\\
\ns\ds \l(s,y)z^2=e^{-y^2}s\cd z^2.\\
\ea$$
\end{example}

\begin{proof}[Proof of Theorem \ref{8.8.21.2}]
In order to solve the equation \rf{8.8.20.25}, let $\p(\cd)$ be as follows,
$$\p(t,y)\triangleq\int_0^y\exp\Big(2\int_0^s\l(t,r)dr\Big)ds,
\qq~\forall~t\in[0,T],~y\in\dbR.$$
Applying It\^{o} formula to $\p(t,Y_t)$ on $[0,T]$, and note that $\p_{yy}(t,y)-2h(t,y)\p_{y}(t,y)=0$, we have
$$\ba{ll}
%
%
\ns\ds \p(t,Y_t)=\p(T,\xi_T)
+\int_t^T\Big(\p_y(r,Y_r)f(r,Y_r,Y_{r+\d(r)},Z_{r+\z(r)})-\p_t(r,Y_r)\Big)dr\\
\ns\ds\qq\qq \ \ \ \ -\int_t^T\p_y(r,Y_r)Z_rdW_r,\qq t\in[0,T].
\ea$$
Now, if we let
$$\bar{Y}_t\triangleq\p(t,Y_t),\q~\bar{Z}_t\triangleq\p_y(t,Y_t)Z_t,\q~ t\in[0,T],$$
then the equation \rf{8.8.20.25} can be transformed formally to the following classical anticipated BSDE:
\bel{3.48}\bar{Y}_t=\bar \xi_T+\int_t^T\of(r,\bar{Y}_r,\bar{Y}_{r+\d(r)},\bar{Z}_{r+\z(r)})dr
-\int_t^T\bar{Z}_rdW_r,\qq t\in[0,T],\ee
where
$$\2n\ba{ll}
\ds\of(s,y,\xi_r,\eta_{\bar r})
=\p_y(s,\p^{-1}(s,y))f\(s,\p^{-1}(s,y),\p^{-1}(r,\xi_r),
\frac{\eta_{\bar r}}{\p_y(s,\p^{-1}(s,y))}\)\\
\ns\ds\qq\qq\qq\q -\p_s(s,\p^{-1}(s,y)),\qq r,\bar r\in[s,T+K],\ s\in[t,T],\\
\ns\ds\qq\q\ \  \bar \xi(s)=\p^{-1}(s,\xi_s), \qq s\in[T,T+K].
\ea$$
From Lemma \ref{8.8.20.34}, $\p_{y}(\cd)$ is a positive bounded function, and both $\p$ and its inverse $\p^{-1}(\cd)$ are smooth functions which are globally Lipschitz, we then deduce that \rf{8.8.20.25} admits an adapted solution if and only if \rf{3.48} admits an adapted solution. Moreover,  under Assumption \ref{8.8.20.49}, it is easy to see that $\of(\cd)$ is Lipschitz continuous in all of it arguments. Furthermore, since $\xi_\cd\in S_\dbF^2(0,T+K;\dbR)$ and $\p^{-1}(\cd)$ is globally Lipschitz, it follows that $ \bar \xi_\cd$ is also belongs to $S_\dbF^2(0,T+K;\dbR)$. Now, from Theorem 4.2 of Peng and Yang \cite{Peng-Yang-09}, the classical anticipated BSDE \rf{3.48} admits a unique adapted solution $(\bar{Y},\bar{Z})\in S_\dbF^2(0,T+K;\dbR)\times L^2_\dbF(0,T+K;\dbR)$, which implies that BSDE \rf{8.8.20.25} also admits unique adapted solution. In fact, the triple $(Y_\cd,Z_\cd)$ defined by
$$Y_t\triangleq \p^{-1}(t,\bar{Y}_t),\qq
Z_t\triangleq\frac{\bar{Z}_)}{\p_y(t,\p^{-1}(t,\bar{Y}_t))},\q~t\in[0,T],$$
is the unique adapted solution of BSDE \rf{8.8.20.25}. This completes the proof.
\end{proof}

\section*{Acknowledgements}

The authors would like to thank Professor Shige Peng for his helpful discussions and comments.

\end{document}